\documentclass[final]{siamltex}

\usepackage{latexsym,amsmath,amssymb,amsfonts,amscd,multirow}
\usepackage{epsfig,verbatim,epstopdf,graphics}
\usepackage{array}
\usepackage[ruled,linesnumbered]{algorithm2e}
\usepackage{mathtools} 

\usepackage{bm,xcolor}	
\usepackage{subfigure}

\newtheorem{Thm}{Theorem}[section]
\newtheorem{Lem}[Thm]{Lemma}

\newtheorem{exa}[Thm]{Example}

\newcommand{\ph}{\varphi}
\newcommand{\ta}{\tau_{\max}}

\newcommand{\ot}{\frac{1}{2}}

\newcommand{\into}{\int_\Omega}

\newcommand{\beq}{\begin{equation}}
\newcommand{\eeq}{\end{equation}}
\newcommand{\beqa}{\begin{eqnarray}}
\newcommand{\eeqa}{\end{eqnarray}}
\newcommand{\bmat}{\begin{bmatrix}}
\newcommand{\emat}{\end{bmatrix}}

\newcommand{\intom}{\int_{\Omega}}

\newcommand{\bv}{{\bf v}}

\newcommand{\bu}{{\bf u}}

\newcommand{\bV}{{\bf V}}
\newcommand{\bn}{{\bf n}}

\newcommand{\bP}{{\bf P}}

\newcommand{\bs}{{\bf s}}

\newcommand{\bzero}{\mathbf{0}}
\newcommand{\bone}{\mathbf{1}}
\newcommand{\bl}{\mathbf{l}}
\newcommand{\bx}{\mathbf{x}}

\newcommand{\bp}{\mathbf{p}}
\newcommand{\bj}{\mathbf{j}}
\newcommand{\bW}{\mathbf{W}}

\newcommand{\bal}{{\bm{\alpha}}}
\newcommand{\bb}{{\bm{\beta}}}
\newcommand{\bq}{\mathbf{q}}

\newcommand{\abs}[1]{\left | #1 \right |}

\newcommand{\cb}{}
\DeclarePairedDelimiter{\floor}{\lfloor}{\rfloor}

\newcolumntype{H}{>{\setbox0=\hbox\bgroup}c<{\egroup}@{}}
\title
{Sparse grid central discontinuous Galerkin method  for linear hyperbolic systems in high dimensions}

\author{
Zhanjing Tao
\thanks{Department of Mathematics, Michigan State University,
East Lansing, MI 48824 U.S.A.
 {\tt taozhanj@msu.edu}}
\and
Anqi Chen
\thanks{Department of Mathematics, Michigan State University,
East Lansing, MI 48824 U.S.A.
{\tt chenanq3@msu.edu}.}%
\and
Mengping Zhang
\thanks{School of Mathematical Sciences, University of Science and Technology of China, Hefei, Anhui, 230026 People's Republic of China. {\tt mpzhang@ustc.edu.cn.} Research supported by NSFC grant 11471305.}
\and
 Yingda Cheng
\thanks{Department of Mathematics, Department of  Computational Mathematics, Science and Engineering, Michigan State University,
East Lansing, MI 48824 U.S.A.
 {\tt ycheng@msu.edu}. Research is supported by NSF grants  DMS-1453661, DMS-1720023 and the Simons Foundation.}
}

\date{\today}

\begin{document}
\maketitle

\begin{abstract}
In this paper, we develop sparse grid central discontinuous Galerkin (CDG) scheme for linear hyperbolic systems  with variable coefficients in high dimensions. The scheme combines the CDG framework with the sparse grid approach, with the aim of breaking the curse of dimensionality. 
A new hierarchical representation of piecewise polynomials on the dual mesh is introduced and analyzed, resulting in a sparse finite element space that can be used for non-periodic problems. Theoretical results, such as $L^2$ stability and error estimates are obtained for scalar problems. CFL conditions are studied numerically comparing discontinuous Galerkin (DG), CDG, sparse grid DG and sparse grid CDG methods.    Numerical results including scalar linear equations, acoustic and elastic waves are provided.
\end{abstract}


\section{Introduction}


In this paper, we develop sparse grid central discontinuous Galerkin (CDG) method for the following time-dependent linear hyperbolic system with variable coefficients 
\begin{equation}
\label{eq:model0}
\frac{\partial \mathbf{u}}{\partial t} + \sum_{i=1}^{d}\frac{\partial (A_i(t, \bx) \mathbf{u})}{\partial x_i} =\mathbf{0},\quad \bx\in\Omega,  
\end{equation}
subject to appropriate initial and boundary conditions. In the expression above, $d \ge 2$ is the spatial dimension of the problem,  $\mathbf{u}(t, \bx) = (u^1(t, \bx), \cdots, u^m(t, \bx))^T$ is the unknown function, $A_i(t, \bx) \in \mathbb{R}^{m \times m}, i=1, \ldots, d$ are the given smooth variable coefficients. We assume $\Omega=[0,1]^d$ in the paper, but the discussion can be easily generalized to arbitrary box-shaped domains.  The model \eqref{eq:model0} arises in many contexts \cite{leveque2002finite}, such as simulations of acoustic, elastic waves,  and Maxwell's equations in free space. The scheme we develop in this paper can also apply to the case when $A_i(t, \bx)$ is defined through another set of equations that can be nonlinearly coupled with $\bu$, such as the models in kinetic plasma waves and incompressible flows. 

Many numerical methods, including finite difference, finite volume, finite element, spectral methods \textit{etc.}, have been developed in the literature for \eqref{eq:model0} addressing different challenges in various applications. The focus of this paper, is to design a class of conservative numerical schemes, with high computational efficiency, for system \eqref{eq:model0} when $d$ is large. It is well known that any grid based method suffers from the curse of dimensionality \cite{bellman1961adaptive}. This term refers to the fact that the computational cost and storage requirements   scale as  $O(h^{-d})$ for a $d$-dimensional problem, where   $h$ denotes the mesh size in one coordinate direction, while the approximation accuracy is independent of $d.$   To overcome this bottleneck,  sparse grid methods \cite{zenger1991sparse, bungartz2004sparse, garcke2013sparse} were introduced to  reduce the degrees
of freedom for   high-dimensional numerical simulations. Sparse grid techniques have been incorporated in  collocation methods for   high-dimensional stochastic differential equations \cite{xiu2005high,xiu2007efficient,nobile2008sparse,ma2009adaptive}, finite element methods \cite{zenger1991sparse, bungartz2004sparse,schwab2008sparse}, finite difference methods \cite{griebel1998adaptive, griebel1999adaptive}, finite volume methods \cite{hemker1995sparse}, and spectral methods \cite{griebel2007sparse, gradinaru2007fourier,shen2010sparse,shen2010efficient} for   high-dimensional PDEs. 

In recent years, we initiate a line of research on the development of sparse grid discontinuous Galerkin (DG) methods \cite{sparsedgelliptic, guo2016sparse,guo2017adaptive}. The DG method \cite{Cockburn_2000_history} is a class
of finite element methods using  discontinuous approximation space for the numerical solution and the test functions. The Runge-Kutta DG scheme \cite{Cockburn_2001_RK_DG} developed in a series of papers for hyperbolic equations became very popular  due to its provable convergence, excellent conservation properties and accommodation for adaptivity and parallel implementations.  The sparse grid DG method designed in \cite{guo2016sparse} is well suited for time-dependent transport problems in high dimensions, reducing degrees of freedom of from $O(h^{-d})$ to $O(h^{-1}|\log_2h|^{d-1})$, maintaining conservation, with provable convergence rate of $O(|\log_2h|^{d}h^{{k+1/2}})$ in $L^2$ norm when the solution is smooth. Similar to  \cite{guo2016sparse}, in this paper, we restrict our attention to smooth solutions of \eqref{eq:model0}. It is known that for non-smooth solutions, adaptivity should be invoked to capture discontinuity like structures. This can be achieved using the idea in \cite{guo2017adaptive} and is left for our future work.

Based on the scheme constructed in \cite{guo2016sparse}, the goal of the present paper is  to design and analyze the sparse grid CDG method. The CDG schemes \cite{liu2005central,liu2007central,liu20082} are a class of DG schemes on overlapping cells that combine  the idea of the central schemes \cite{nessyahu1990non,kurganov2000new,liu2005central2} with the DG weak formulation. Such methods are intrinsically Riemann solver free, therefore no costly flux evaluations are needed in the computation. It is well known that the CDG schemes allow   larger CFL numbers than the standard DG methods except for piecewise constant approximations  \cite{liu2007central,reyna2015operator}. This compensates the increased cost caused by duplicate representation of the solution on the dual mesh. Motivated by this, we develop sparse grid CDG method that avoids the evaluation of numerical fluxes.   We investigate stability, convergence rate and CFL condition of the resulting scheme. A  novelty of this work is the design of the scheme for non-periodic problems, where a new hierarchical representation of the solution is presented, which results in a sparse finite element space that can be defined on the dual mesh. $L^2$ projection results are studied for this space, which helps the convergence proof of the schemes for initial-boundary value problems.  
  
 The rest of this paper is organized as follows:
in Section \ref{sec:method},  we  construct the sparse grid CDG formulations for periodic and non-periodic problems, and perform numerical study of the CFL conditions.  In Section \ref{sec:analysis}, we prove $L^2$ stability and error estimates for scalar  equations. The numerical performance is validated in Section \ref{sec:numerical} by   several benchmark tests, including scalar transport equations, acoustic and elastic waves. Conclusions and future work are discussed in Section \ref{sec:conclu}.

\section{Numerical methods}
\label{sec:method}

In this section, we   define and discuss the properties of the proposed sparse grid CDG methods. For convenience of notations, we
  rewrite \eqref{eq:model0} in a component-wise form as
\begin{equation}
\label{eq:model}
\frac{\partial u^l}{\partial t} + \nabla \cdot (A^l(t, \bx)  \mathbf{u} ) =0, \quad l = 1, \cdots, m, \quad \bx\in\Omega,
\end{equation}
where $A^l(t, \bx) = (A_1^l(t, \bx), \cdots, A_{d}^l(t, \bx))^T \in \mathbb{R}^{d \times m}$ denotes a collection of the $l$-th row of each matrix $A_i$. The problem is solved with given initial value $\mathbf{u}(0,\bx) = \mathbf{u}_0(\bx),$ and periodic or Dirichlet type boundary conditions.

We proceed as follows. First, we introduce the scheme for periodic problems. In this setting, the finite element space on the primal and dual mesh can be defined in similar ways. Then, we discuss the implementation details and perform numerical study of the CFL conditions. Finally, we consider the more complicated non-periodic problems, for which a new sparse  finite element space will be introduced on the dual mesh.

\subsection{Periodic problems}
\label{sec:periodic}

To define the sparse   finite element space, we first review the hierarchical decomposition of piecewise polynomial space in one dimension \cite{sparsedgelliptic}. Consider a general interval $[a,b]$,  we define the $n$-th level  mesh $\Omega_{n}([a,b])$ to be a uniform partition of $2^n$  cells with length $h_n = 2^{-n}(b-a)$ and $I_{n }^j =[a + j h_n, a + (j+1) h_n]$, $j=0, \ldots, 2^n-1,$ for any $n \ge 0.$ 
Let
$$V_{n }^k([a,b]):=\{v: v \in P^k(I_{n }^j),\, \forall \,j=0, \ldots, 2^n-1\}$$
be the usual piecewise polynomials of degree at most $k$ on   $\Omega_{n }$. Then, we have the nested structure $$V_{0 }^k([a,b]) \subset V_{1 }^k([a,b]) \subset V_{2 }^k([a,b]) \subset V_{3 }^k([a,b]) \subset  \cdots$$

Similar to \cite{sparsedgelliptic},  we can now define the multiwavelet subspace $W_{n }^k([a,b])$, $n=1, 2, \ldots $ as the orthogonal complement of $V^k_{n-1}([a,b])$ in $V^k_{n}([a,b])$ with respect to the $L^2$ inner product on $[a,b]$, i.e.,
\begin{equation*}
V_{n-1 }^k([a,b]) \oplus W_{n }^k([a,b])=V_{n }^k([a,b]), \quad W_{n }^k([a,b]) \perp V_{n-1 }^k([a,b]).
\end{equation*}
For notational convenience, we let
$W_{0 }^k([a,b]):=V_{0 }^k([a,b])$, which is the standard piecewise polynomial space of degree $k$ on $[a,b]$.  This gives the hierarchical decomposition  $V_{n }^k([a,b])$ on $\Omega_{n }$ as $V_{n }^k([a,b])=\bigoplus_{0 \leq l\leq n} W_{l }^k([a,b])$.

For a $d$ dimensional domain $[a,b]^d,$ we recall some basic notations about multi-indices. For a multi-index $\mathbf{\alpha}=(\alpha_1,\cdots,\alpha_d)\in\mathbb{N}_0^d$, where $\mathbb{N}_0$  denotes the set of nonnegative integers, the $l^1$ and $l^\infty$ norms are defined as 
$$
|\bal|_1:=\sum \nolimits_{i=1}^d \alpha_i, \qquad   |\bal|_\infty:=\max_{1\leq i \leq d} \alpha_i.
$$
The component-wise arithmetic operations and relational operations are defined as
$$
\bal \cdot \bb :=(\alpha_1 \beta_1, \ldots, \alpha_d \beta_d), \qquad c \cdot \bal:=(c \alpha_1, \ldots,  c \alpha_d), \qquad 2^\bal:=(2^{\alpha_1}, \ldots, 2^{\alpha_d}),
$$
$$
\bal \leq \bb \Leftrightarrow \alpha_i \leq \beta_i, \, \forall i,\quad
\bal<\bb \Leftrightarrow \bal \leq \bb \textrm{  and  } \bal \neq \bb.
$$

By making use of the multi-index notation, we denote by $\bl=(l_1,\cdots,l_d)\in\mathbb{N}_0^d$ the mesh level in a multivariate sense. We  define the tensor-product mesh grid $\Omega_{\bl }([a,b]^d)=\Omega_{l_1 }([a,b])\otimes\cdots\otimes\Omega_{l_d }([a,b])$ and the corresponding mesh size $h_\bl=(h_{l_1},\cdots,h_{l_d}).$ Based on the grid $\Omega_{\bl }$, we denote by $I_{\bl }^\bj=\{\bx:x_i\in I_{l_i}^{j_i}, i=1,\cdots,d\}$ as an elementary cell, and 
$$\bV_{\bl }^k ([a,b]^d):=\{v: v(\bx) \in Q^k(I^{\bj}_{\bl }), \,\,  \bzero \leq \bj  \leq 2^{\bl}-\bone \}= V^k_{l_1,x_1 }([a,b])\times\cdots\times  V_{l_d,x_d }^k([a,b])$$
as the standard tensor-product piecewise polynomial space on this mesh, where $Q^k(I^{\bj }_{\bl})$ denotes the collection of polynomials of degree up to $k$ in each dimension on cell $I^{\bj}_{\bl }$. 
If $\bl = (N,\cdots, N)$, the  grid and space will be  further denoted by  $\Omega_N([a,b]^d)$ and $\bV_{N }^k([a,b]^d)$, respectively.

Based on a tensor-product construction, the multi-dimensional increment space can be  defined as
$$\bW_\bl^k([a,b]^d)=W_{l_1,x_1}^k([a,b])\times\cdots\times  W_{l_d,x_d}^k([a,b]).$$ 
Therefore, we have  
$ 
\bV_{N }^k ([a,b]^d)  = \bigoplus_{\substack{ |\bl|_\infty \leq N\\\bl \in \mathbb{N}_0^d}} \bW_{\bl }^k([a,b]^d).
$ 
The sparse finite element approximation space we consider,    is  defined by
$$ \hat{\bV}_{N }^k([a,b]^d):=\bigoplus_{\substack{ |\bl|_1 \leq N\\\bl \in \mathbb{N}_0^d}}\bW_{\bl }^k([a,b]^d).$$
This is a subset of   $\bV_{N }^k([a,b]^d)$, and   its number of degrees of freedom  scales as $O((k+1)^d2^NN^{d-1})$ \cite{sparsedgelliptic}, which is significantly less than that of $\bV_N^k([a,b]^d)$ with exponential dependence on $Nd$. This is the key to computational savings in high dimensions.


The standard CDG schemes  \cite{liu2005central,liu2007central} is characterized by   numerical approximations on two sets of overlapping grids: primal and dual meshes. Now, we are ready to incorporate the sparse  finite element space defined above into the CDG framework. For the domain under consideration $\Omega=[0,1]^d,$ we let
  $\Omega_{N,P} := \Omega_{N}([0,1]^d)$ be the primal mesh and  $\Omega_{N,D}$, which is the periodic extension of $\Omega_{N}([-h_N/2,1-h_N/2]^d)$ restricted to $[0,1]^d$, be the dual mesh. Similarly, we let $\hat \bV_{N,P}^k := \hat \bV_{N}^k([0,1]^d)$ and $\hat \bV_{N,D}^k$ to be the periodic extension of $\hat \bV_{N}^k([-h_N/2,1-h_N/2]^d)$ restricted to $[0,1]^d$. Here and below, the subscripts $P$ and $D$ represent the quantities defined on the primal and dual mesh, respectively. 
  
  The approximation properties for the sparse finite element space have been established in previous work \cite{sparsedgelliptic,guo2016sparse}. By using a lemma  in \cite{guo2016sparse}, we can have  estimates for $L^2$ projection operator onto the spaces $\hat \bV_{N,P}^k,  \hat \bV_{N,D}^k.$ 
  
To facilitate the discussion, below we introduce some notations about norms and semi-norms. Let $G = P,D$, on primal or dual mesh $\Omega_{N,G}$,
we use $\|\cdot\|_{H^s(\Omega_{N,G})}$  to denote the standard broken Sobolev norm, i.e. $\|v\|^2_{H^s(\Omega_{N,G})}=\sum_{\bzero \le \bj \le 2^\mathbf{N}-\mathbf{1}} \|v\|^2_{H^s(I_{N,G}^\bj)},$ where $\|v\|_{H^s(I_{N,G}^\bj)}$ is the standard Sobolev norm on $I_{N,G}^\bj,$ (and $s=0$ is used to denote the $L^2$ norm). Similarly,  we use $|\cdot|_{H^s(\Omega_{N,G})}$ to  denote the broken Sobolev semi-norm, and $\|\cdot\|_{H^s(\Omega_{\bl,G})}, |\cdot|_{H^s(\Omega_{\bl,G})}$ to denote the broken Sobolev norm and semi-norm that are supported on a general grid $\Omega_{\bl,G}$.
For any set $L=\{i_1, \ldots i_r \} \subset \{1, \ldots d\}$, we define $L^c$ to be the complement set of $L$ in $\{1, \ldots d\}.$ For a non-negative integer $\alpha$ and set $L$,  we define the semi-norm on any domain denoted by $\Omega'$
\begin{flalign*}
|v|_{H^{\alpha,L}(\Omega')} :=  \left \| \left ( \frac{\partial^{\alpha}}{\partial x_{i_1}^{\alpha}} \cdots \frac{\partial^{\alpha}}{\partial x_{i_r}^{\alpha}}  \right ) v \right \|_{L^2(\Omega')},
\end{flalign*}
and
$$
|v|_{\mathcal{H}^{q+1}(\Omega')} :=\max_{1 \leq r \leq d} \left ( \max_{\substack{L\subset\{1,2,\cdots,d\} \\|L|=r}} |v|_{H^{q+1, L}(\Omega')} \right ),$$
which is the norm for the mixed derivative of $v$ of at most degree $q+1$ in each direction.
In this paper, we use the notation $A \lesssim B$  to represent $A \leq \text{constant}\times B$, where the constant  is independent of $N$ and the mesh level considered. 
The following results are obtained from  Lemma 3.2 in \cite{guo2016sparse}.

 \begin{Lem}[$L^2$ projection estimate]
\label{lem:ltproj}
Let $\mathbf{P}_P, \mathbf{P}_D$  be     $L^2$ projections  onto the spaces $\hat{\bV}_{N,P}^k, \hat{\bV}_{N,D}^k$, respectively, then for  $k \geq 1$,  $1 \leq q \leq \min \{p, k\}$, and  
$v \in \mathcal{H}^{p+1}(\Omega),$ which is periodic on $\Omega$, $N\geq 1$, $d \geq 2$, we have for $G=P, D,$
\begin{align}
\label{eq:plterr}
  | \bP_G v- v |_{H^s(\Omega_{N,G})}\lesssim \begin{cases}  N^d  2^{-N(q+1)} |v|_{\mathcal{H}^{q+1}(\Omega)}& s=0,\\
  2^{-Nq} |v|_{\mathcal{H}^{q+1}(\Omega)}& s=1.
  \end{cases}
  \end{align}
\end{Lem}
This lemma shows that the $L^2$ norm and $H^1$ semi-norm of the projection error scale  like $O(N^d 2^{-N(k+1)})$  and $O(2^{-Nk})$ with respect to $N$ when the function $v$ has bounded mixed derivatives up to  enough degrees. This lemma will be used in Theorem \ref{thm:lterr} to establish convergence of the scheme.

%
%
%

\bigskip

Now, we are ready to formulate the sparse grid CDG scheme. Below we review some standard notations about jumps and averages of piecewise  functions. 
With $G=P$ or $D$, let $T_{h,G}$ be the collection of all elementary cell $I^{\bj}_{N,G}$,
$\Gamma_{N,G}:=\bigcup_{T \in \Omega_{N,G}} \partial T$ be the union of the interfaces for all the elements in $\Omega_{N,G}$ (here we have taken into account the periodic boundary condition when defining $\Gamma_{N,G}$) and $S(\Gamma_G):=\Pi_{T\in \Omega_{N,G}} L^2(\partial T)$ be the set of $L^2$ functions defined on $\Gamma_{N,G}$. For any $q \in S(\Gamma_{N,G})$ and $\bq \in [S(\Gamma_{N,G})]^d$,  we define their   averages $\{q\}, \{\bq\}$ and jumps $[q], [\bq]$ on the interior edges as follows. Suppose $e$ is an interior edge shared by elements $T_+$ and $T_-$, either on primal or dual mesh, we define the unit normal vectors $\bm{n}^+$ and  $\bm{n}^-$ on $e$ pointing exterior of $T_+$ and $T_-$, respectively, then
\begin{flalign*}
[ q] \  =\  \, q^- \bm{n}^- \, +  q^+ \bm{n}^+, & \quad \{q\} = \frac{1}{2}( q^- + 	q^+), \\
[ \bq] \  =\  \, \bq^- \cdot \bm{n}^- \, +  \bq^+ \cdot \bm{n}^+, & \quad \{\bq\} = \frac{1}{2}( \bq^- + \bq^+).
\end{flalign*}

The semi-discrete sparse grid CDG scheme for \eqref{eq:model}, based on the weak formulation introduced in \cite{liu2005central,liu2007central},  is defined as follows: we find $u^l_h\in\hat{\bV}_{N,P}^k$ and $v^l_h\in\hat{\bV}_{N,D}^k$, such that $\forall \,l = 1, \cdots, m$
\begin{align}  
\label{eq:DGformulation_p}
\int_{\Omega}(u^l_h)_t\,\ph_h\,d\bx =& \frac{1}{\ta}\int_{\Omega} (v^l_h-u^l_h) \,\ph_h\,d\bx  + \int_{\Omega} A^l(t, \bx) \mathbf{v}_h \cdot\nabla \ph_h\,d\bx - \sum_{\substack{e \in \Gamma_{N,P}}}\int_{e} A^l(t, \bx) \mathbf{v}_h\cdot  [\ph_h]   \,ds, \\
\label{eq:DGformulation_d}
\int_{\Omega}(v^l_h)_t\,\psi_h\,d\bx =& \frac{1}{\ta}\int_{\Omega} (u^l_h-v^l_h) \,\psi_h\,d\bx +  \int_{\Omega}  A^l(t, \bx) \mathbf{u}_h \cdot\nabla \psi_h\,d\bx - \sum_{\substack{e \in \Gamma_{N,D}}}\int_{e} A^l(t, \bx) \mathbf{u}_h \cdot [\psi_h]  \,ds,
\end{align}
for any $\ph_h \in \hat{\bV}_{N,P}^k$ and $\psi_h \in \hat{\bV}_{N,D}^k$, where $\bu_h = (u_h^1, \cdots, u_h^m), \bv_h = (v_h^1, \cdots, v_h^m)$ and $\ta$ is an upper bound for the time step due to the CFL restriction (see Section \ref{sec:cfl} for detailed discussions).


{\cb
\subsection{Discussions on implementations} Here, we briefly discuss some details about the implementation of the scheme.
We perform the computation by using orthonormal multiwavelet bases  constructed by Alpert \cite{alpert1993class}. In 1D, the bases of $W_{l }^k([0,1])$ are denoted by
$$
v_{p, l }^j(x), \quad p = 1,\cdots, k+1, \quad j = 0, \cdots, 2^{l-1}-1
$$
and they satisfy $\int_{a}^b v_{p, l }^j(x)v_{p', l' }^{j'} (x) dx = \delta_{p p'} \delta_{j j'} \delta_{l l'}$. Figures  \ref{fig:nonp_basis_k01}  and \ref{fig:nonp_basis_k11}  provide illustrations of the basis functions for $k=0, 1$ and $l=0,1,2.$ The bases in  $W_{\bl }^k$  in multi-dimensions are defined by tensor products
$$
v_{\bs}=v_{\bp,\bl }^\bj := \prod_{i=1}^d v_{p_i,l_i }^{j_i} (x_i), \quad p_i = 1, \cdots, k+1, \quad j_i = 0, \cdots, \max(0, 2^{l_i-1}-1),
$$
where we have used the notation $\bs=(\bl, \bj, \bp)$ and $s_i=(l_i, j_i, p_i)$ to denote the multi-index for the bases.

As for temporal schemes, we can use the  total variation diminishing  Runge-Kutta (TVD-RK) methods \cite{Shu_1988_JCP_NonOscill} to solve the ordinary differential  equations for the coefficients resulting from the   discretization.   
To calculate the right-hand-side of  \eqref{eq:DGformulation_p}-\eqref{eq:DGformulation_d}, the fast matrix-vector product by  LU split or LU decomposition algorithms \cite{shen2010efficient, shen2012efficient, rong2017nodal} can be applied, by which one can decompose all calculations into one dimensional operations. 
Below, we briefly describe the LU decomposition algorithm  for the calculation of the following matrix-vector product which appears at the right-hand-side of  \eqref{eq:DGformulation_p}-\eqref{eq:DGformulation_d}
$$b_{\bf j} = \sum_{\bs: |\bl|_1\le N} f_{\bf s} t_{s_1,j_1}^1 \cdots t_{s_d,j_d}^d,$$       
where $f_{\bf s}$ can be the coefficient of the basis   in sparse grid space and
$t_{s_i,j_i}^i, \, i=1, \cdots, d,$ are the corresponding one-dimensional transform of coefficients from basis $v_{s_i}$ to basis $v_{j_i}$ in the $i$-th dimension in our scheme. 
Note that we have $n=2^N (k+1)$ one-dimensional bases in each dimension, and  we use $v_{s_i}$ to denote the $s_i$-th basis. The bases are ordered according to grid increment. 
Using Algorithm 1 in \cite{shen2012efficient}, we should  calculate all the one-dimensional transform along each direction associated with a block lower triangular matrix, and then calculate all the one-dimensional transforms having a block upper triangular structure. 
 The fast matrix-vector product $f_{\bf s} \rightarrow b_{\bf j} $ on sparse grid with LU decomposition can be proceeded as follows.
\begin{enumerate}
	\item Calculate (block) LU decomposition 
	$t_{s,j}^{i} = \sum_{m=1}^{n} (Pl)^i_{s,m}(uQ)^i_{m,j}, \, s,j= 1, \cdots, n,$ for $i=1, \cdots, d,$ where   $P^i, Q^i$ are the permutation matrices, $l^i, u^i$ are lower and upper triangular matrices.
	\item Compute the transform with a (block) lower triangular matrix for $i=1, \cdots, d,$
	
	$b_{s_1,\cdots, s_{i-1},s_i^{'},s_{i+1},\cdots, s_{d}} \leftarrow \sum_{s_i: l_1+ \cdots +l_{d} \le N} f_{\bf s} (Pl)_{s_i,s_i^{'}}^i.$  	
	\item Compute the transform with a (block) upper triangular matrix for $i=1, \cdots, d,$
	
	$ b_{\bf s} \leftarrow  \sum_{s_i^{'}: l_1+ \cdots +l_{i-1} + l_i^{'} +l_{i+1} + \cdots +l_{d}  \le N} b_{s_1,\cdots, s_{i-1},s_i^{'},s_{i+1},\cdots, s_{d}} (uQ)_{s_i^{'},s_i}^i.$

\end{enumerate}       

Note that in step 1, the LU decomposition pivots only from rows or columns in the same mesh level
to maintain the hierarchical structure. This pivoting can be successfully done in the sparse grid CDG scheme, but not in the sparse grid DG scheme, for which additional splitting of the flux terms are deemed necessary for variable coefficient case.

For the integrals involving variable-coefficient, we use Gaussian quadrature to compute these terms. Since these integrals are
multi-dimensional integrations, we use the so-called unidirectional principle to separate the integration into multiplication of one-dimensional integrals. For example, if $\phi(x)=\phi_1(x_1) \cdots\phi_d(x_d)$ is separable,
$$\int_{\Omega}\phi(x) = \int_{[a,b]}\phi_1(x_1) \cdots \int_{[a,b]}\phi_d(x_d).$$
When the variable coefficient $A_i(t,x)$ is separable, we can use unidirectional principle directly. If it is not separable, we can find $A_i^h(t,x)$ as the $L^2$ projection of $A_i(t,x)$ onto the sparse grid finite element space, and then use  $A_i^h(t,x)$ to compute the integrals. 

}

\subsection{Discussions on CFL conditions}
\label{sec:cfl}

It is well known that the CDG schemes allow   larger CFL numbers than the standard DG methods except for piecewise constant approximations  \cite{liu2007central,reyna2015operator}. Here, we perform a numerical study of the  CFL conditions of DG \cite{Cockburn_2001_RK_DG}, CDG \cite{liu20082}, sparse grid DG \cite{guo2016sparse}, and the sparse grid CDG schemes. We only consider the two-dimensional case solving constant coefficient equation $u_t+u_{x_1}+u_{x_2}=0$ for now. The results are listed in Table \ref{table:linear_cfl}. The CFL number of DG method is obtained from Table 2.2 in \cite{Cockburn_2001_RK_DG}. The rest of the table is computed by eigenvalue analysis of the discretization matrix, and by requiring the amplification of the   eigenvalues to be bounded by 1 in magnitude. We observe
that the sparse grid DG method has  CFL
number that is about two times the CFL number of the standard DG method. The sparse grid CDG method offers the largest CFL conditions among all four methods. Here, as a side note, we find that the CFL number for two-dimensional CDG method is larger than the CFL number for one-dimensional CDG method in  \cite{liu20082}. This table shows that one advantage of   the sparse grid CDG method is the ability to take large time steps for time evolution problems. In general, further numerical results suggest that for equation    $u_t+c_1u_{x_1}+c_2u_{x_2}=0,$ the CFL number for sparse grid DG and sparse grid CDG method will change with the value of the coefficients $c_1, c_2.$  Results in higher dimensions are yet to be studied. A preliminary calculation shows that for equation  $u_t+u_{x_1}+u_{x_2}+u_{x_3}=0$ the CFL conditions for CDG,  sparse grid DG and sparse grid CDG methods in 3D are all higher than those for the 2D case in Table \ref{table:linear_cfl}. The sparse grid CDG method still possesses the largest CFL number among all four methods.  Those interesting issues will be investigated in our future work.

\begin{table}
	\caption{CFL numbers of the   DG method, CDG method, sparse grid DG method and sparse grid CDG method with piecewise degree $k$ polynomials, Runge-Kutta method of order $\nu$ for Example \ref{ex:linear} with d=2. The CFL numbers of the sparse grid DG/CDG methods are measured with regard to the most refined mesh $h_N.$}
	\centering
	\begin{tabular}{|c |c c c|  c  c c|  c  c c|  c  c c|}
		
		\hline
		
		&        \multicolumn{3}{c|}{DG}          &     \multicolumn{3}{c|}{ CDG}    
		&     \multicolumn{3}{c|}{sparse grid DG}   &     \multicolumn{3}{c|}{sparse grid CDG}     \\
		\hline
		$k$ & 1& 2 & 3&  1 & 2 & 3 & 1& 2 & 3 & 1& 2 & 3 \\
		\hline
		$\nu=2$&0.33  & --	&	--   &	0.48	 &--	& -- & 0.66  & --	&	--   &	0.87	 &--	& --	 	\\
		\hline
		$\nu=3$&0.40  & 0.20	&	0.13   &	0.66	 & 0.36	& 0.24 & 0.81  & 0.41	&	0.25   &	1.17	 & 0.65	& 0.44 	 	\\
		\hline
		$\nu=4$&0.46  & 0.23	& 0.14   &	0.90	 & 0.52	& 0.35 & 0.92  & 0.46	& 0.28   &	1.58	 & 0.94	& 0.62	 	\\
		\hline
		
	\end{tabular}
	\label{table:linear_cfl}
\end{table}

\subsection{Non-periodic problems}
Here, we consider non-periodic problems, where  equation \eqref{eq:model0} or \eqref{eq:model} is supplemented by   Dirichlet boundary condition on the inflow edges.  In this case, we can no longer use periodicity to define the finite element space on the dual mesh, and a new grid hierarchy needs to be introduced.

  Recall that for standard CDG methods with non-periodic boundary condition  on the domain $[0,1],$ the finite element space on dual mesh with cell size $h_n=1/2^n$ is represented by
\begin{equation}
\label{eq:sd}
 V_{n,D}^k= \{v: v \in P^k(I_{n,D}^{j}),\, \forall \,j=0, \ldots, 2^n \},
 \end{equation}
where the mesh is partitioned as $$I^0_{n, D}=[0,\frac{h_n}{2}], \quad I^j_{n,D}=[(j-\ot)h_n, (j+\ot)h_n], \,j=1, \ldots, 2^n-1, \quad I^{2^n}_{n,D}=[1-\frac{h_n}{2}, 1],$$
which consists of $2^n-1$ cells of size $h_n,$  and two cells at the left and right ends of size $h_n/2.$ It is easy to see that this space does not  have nested structures, i.e. $V_{n-1,D}^k \not \subset V_{n,D}^k.$ Therefore, we need a new hierarchy to define the increment polynomial spaces.

 For a fixed refined mesh level $N,$ we define the following  grid $\Omega_{l,N,D}$ on level $l, \,l=0 \ldots N,$ by a collection of cells as
$$
I^0_{l,N,D}=[0,h_l-\frac{h_N}{2}], \quad I^j_{l,N,D}=[j h_l-\frac{h_N}{2}, (j+1)h_l-\frac{h_N}{2}], \,j=1, \ldots, 2^l-1, \quad I^{2^l}_{l,N,D}=[1-\frac{h_N}{2}, 1],
$$
which consists of $2^l-1$ cells of size $h_l,$  and a cell at the left end of size $h_l-\frac{h_N}{2},$ and a cell at the right end of size $\frac{h_N}{2}.$ This grid structure is naturally nested, and therefore $V^k_{l,N,D}$ which consists of piecewise polynomials of degree $k$ defined on $\Omega_{l,N,D}$ are also nested, and $V^k_{N,N,D}=V_{N,D}^k$ as defined in \eqref{eq:sd}.

Then the definitions of sparse   finite element space in Section \ref{sec:periodic} can be naturally extended here.  We let $W_{l,N,D}^k$, $l=1, \ldots N $ be a complement set of $V_{l-1,N,D}^k$ in $V_{l,N,D}^k,$ i.e.  \begin{equation*}
V_{l-1,N,D}^k \oplus W_{l,N,D}^k=V_{l,N,D}^k.
\end{equation*}
However, we no longer require $W_{l,N,D}^k$ to be $L^2$ orthogonal to $V_{l-1,N,D}^k$, because such definition will be difficult to implement in practice. Instead, we define $W_{l,N,D}^k$ to be a span of basis functions that are shifted basis functions of $W_l^k$ space defined in Section \ref{sec:periodic}, namely,
 $$
 W_{l,N,D}^k = W_l^k([-\frac{h_N}{2},1-\frac{h_N}{2}]) \big \rvert_{[0,1]}, \quad l \geq 1.
$$
By denoting $W_{0,N,D}^k=V_{0,N,D}^k,$ we have decomposed $V^k_{N,D}=\bigoplus_{0 \leq l\leq N} W_{l,N,D}^k.$ Illustration of basis functions by such definitions for $k=0, 1$ and $l=0, 1, 2$ can be found in Figures \ref{fig:nonp_basis_k02} and \ref{fig:nonp_basis_k12}. The dimension of $W_{0,N,D}^k$ is $2(k+1),$ while the dimensions of $W_{l,N,D}^k, \, l=1, \ldots N$ are $2^{l-1}(k+1).$

Finally, the sparse  finite element space on the dual mesh of domain $[0,1]^d$ is defined as
$$  \hat {\tilde {\bV}}_{N,D}^k:=\bigoplus_{\substack{ |\bl|_1 \leq N\\\bl \in \mathbb{N}_0^d}}\bW_{\bl,N,D}^k,$$
where $\bW_{\bl,N,D}^k=W_{l_1,N,D,x_1}^k\times\cdots\times  W_{l_d,N,D,x_d}^k.$ This is a subset of  the full grid space $\bV_{N,D}^k=\bigoplus_{\substack{ |\bl|_\infty \leq N\\\bl \in \mathbb{N}_0^d}}\bW_{\bl,N,D}^k$, and its number of degrees of freedom  scales as $O(2^{d-1}(k+1)^d2^NN^{d-1})$ (the proof is similar to Lemma 2.3 in \cite{sparsedgelliptic}), which is {\cb larger than that of $\hat \bV_{N,P}^k$,} but still significantly less than that of $\bV_{N,D}^k$ with exponential dependence on $Nd$.


We will now investigate the approximation property of the space $ \hat {\tilde {\bV}}_{N,D}^k.$ We can obtain the following result, which essentially states that the $L^2$ projection onto this newly constructed space has the same order of accuracy as $\mathbf{P}_P, \mathbf{P}_D$ in Lemma \ref{lem:ltproj}.

\begin{Lem} [$L^2$ projection estimate  onto  $ \hat {\tilde {\bV}}_{N,D}^k$ ]
\label{lem:l22}
Let $\mathbf{\tilde P}_D$  be   the   $L^2$ projection  onto the space $\hat{ \tilde \bV}_{N,D}^k$, then for  $k \geq 1$, $1 \leq q \leq \min \{p, k\}$, and 
$v \in \mathcal{H}^{p+1}(\Omega)$, $N\geq 1$, $d \geq 2$,  we have
\beq
  \label{eq:qlterr}
    | \mathbf{\tilde P}_D v- v |_{H^s(\Omega_{N,D})}\lesssim \begin{cases}  N^d  2^{-N(q+1)} |v|_{\mathcal{H}^{q+1}(\Omega)}& s=0,\\
  2^{-Nq} |v|_{\mathcal{H}^{q+1}(\Omega)}& s=1.
  \end{cases}
\eeq
\end{Lem}

\begin{proof}
The proof follows same procedure as Appendix A in  \cite{guo2016sparse}. We will mainly highlight the difference in the proof (see Steps 1 and 2 below). The main difference lies in the fact that all the hierarchical spaces (and associated projections)   have dependence not only on $l,$ but also on the finest mesh level $N.$

\textbf{Step 1}: Decomposition of $\mathbf{\tilde P}_D$  into tensor products of one-dimensional increment projections.
We denote $P_{l,N,D}^k$ as the standard $L^2$ projection operator from $L^2([0,1])$ to $V_{l,N,D}^k$, and the induced increment projection
$$
Q_{l,N,D}^k : = \begin{cases} P_{l,N,D}^k - P_{l-1,N,D}^k, & \text{if} \ l=1, \ldots N, \\		P_{0,N,D}^k, 	& \text{if} \ l = 0,\end{cases}
$$
and further denote
$$
 {\tilde{\mathbf{P}}}_{N,D}^k := \sum_{\substack{ |\bl|_1\leq N \\ \bl \in \mathbb{N}_0^d}} Q_{l_1,N,D,x_1}^k \otimes \cdots \otimes Q_{l_d,N,D,x_d}^k,
$$
where the last subindex of ${Q}^k_{l_i, N, D,x_i}$ indicates that the increment operator is defined in $x_i$-direction. We can verify that $\mathbf{\tilde P}_D=  {\tilde{\mathbf{P}}}_{N,D}^k.$ In fact, for any $v$, it's clear that $ {\tilde{\mathbf{P}}}_{N,D}^k v\in  \hat {\tilde {\bV}}_{N,D}^k$. Therefore, we only need
\begin{equation}
\label{eq:portho}
\int_\Omega ({\tilde{\mathbf{P}}}_{N,D}^k v -v) w \, d\bx=0, \qquad \forall \, w \in  \hat {\tilde {\bV}}_{N,D}^k.
\end{equation}
It suffices to show \eqref{eq:portho} for $v \in C^\infty(\Omega)$ which is a dense subset of $L^2(\Omega)$. In fact, we have 
 $$v = \bP_{N,D}^k v + v - \bP_{N,D}^k v,$$ where $\bP_{N,D}^k = P_{N,N,D,x_1}^k \otimes \cdots \otimes P_{N,N,D,x_d}$ is the $L^2$ projection onto the full grid space $\bV_{N,D}^k.$ Therefore,
\begin{align*}
\into  ({\tilde{\mathbf{P}}}_{N,D}^k v - v)w d\bx & = \into  ({\tilde{\mathbf{P}}}_{N,D}^k v - \bP_{N,D}^k v)w d\bx + \into  (v - \bP_{N,D}^k v)w d\bx 	\\
& = -\into (\sum_{\substack{  |\bl|_{\infty} \leq N, |\bl|_1>N \\ \bl \in \mathbb{N}_0^d}} {Q}^k_{l_1, N, D, x_1} \otimes \cdots \otimes {Q}^k_{l_d, N, D, x_d} v  ) \,w \, d\bx.
\end{align*}
The last term in the first row of the equality above vanishes because $w \in \hat {\tilde {\bV}}_{N,D}^k \subset \bV_{N,D}^k.$
  In addition, for any $l \geq 1$, $\phi \in L^2([0,1]), \varphi \in V_{l-1,N,D}^k$
$$
\int_{[0,1]} Q_{l,N,D}^k \phi \, \varphi dx = \int_{[0,1]} (I-P_{l-1,N,D}^k) \phi \, \varphi dx-\int_{[0,1]} (I-P_{l,N,D}^k) \phi \, \varphi dx = 0,
$$
Therefore, by properties of the tensor product projections
$$
\into  ({\tilde{\mathbf{P}}}_{N,D}^k v - v)w d\bx = 0, \quad \forall w\in \hat {\tilde{\bV}}_{N,D}^k,
$$
and the proof for $\mathbf{\tilde P}_D=  {\tilde{\mathbf{P}}}_{N,D}^k$ is complete.

\textbf{Step 2}: Estimation of the increment projections.
For a function $v \in  H^{p+1}([0,1])$, we have the convergence property of the $L^2$ projection $P_{l,N,D}^k$ as follows: for any integer $q$ with $1 \leq q \leq \min\{p, k\}$, $s=0,\,1$,
\begin{align*}
\label{eq:proj1a}
| P_{l,N,D}^k v- v|_{H^s(I^j_{l,N,D})} 	& \leq c_{k,s,q}  (h^j_{l,N})^{(q+1-s)} |v|_{H^{q+1}(I^j_{l,N,D})},\quad j=1,\cdots,2^{l}-1,
\end{align*}
where the mesh size $h_{l,N}^j = \begin{cases} h_l - h_N/2,	& j = 0 \\h_{l}, & j = 1, \cdots, 2^l -1, 	\\ h_N/2, 	& j = 2^l. \end{cases}$

The estimation above directly applies for $Q_{0,N,D}^k = P_{0,N,D}^k$. For $l \geq 1$,  by simple algebra, we have
\begin{align*}
| Q_{l,N,D}^k v  |_{H^s(I^j_{l,N,D})} & \leq \tilde c_{k,s,q}  2^{-l(q+1-s)} |v|_{H^{q+1}(I^{\floor{j/2}}_{l-1,N,D})},	 \quad j = 2, \cdots, 2^l - 1,\\
| Q_{l,N,D}^k v  |_{H^s(I^j_{l,N,D})} & \leq c_{k,s,q}  (h_l)^{(q+1-s)} |v|_{H^{q+1}(I^j_{l,N,D})} + c_{k,s,q} (h_{l-1}-h_N/2)^{(q+1-s)} |v|_{H^{q+1}(I^0_{l-1,N,D})}, \quad j= 0,1, 		\\
					& < \tilde c_{k,s,q}  2^{-l(q+1-s)} |v|_{H^{q+1}(I^0_{l-1,N,D})},	\\
| Q_{l,N,D}^k v  |_{H^s(I^{2^l}_{l,N,D})} & =0,
\end{align*}
with $\tilde c_{k,s,q} = c_{k,s,q} (1 + 2^{q+1-s})$. 

The rest of the proof is then very similar to Appendix A in  \cite{guo2016sparse}, and is omitted.
\end{proof}


\bigskip
We now provide a numerical validation of Lemma \ref{lem:l22} by considering the   error of projection $\mathbf{\tilde P}_D$ for a smooth function 
\begin{equation}
\label{ufunc}
u(\mathbf{x})=\exp\left(\prod_{i=1}^d x_i\right),\quad \mathbf{x}\in[0,1]^d.
\end{equation}
In Table \ref{table:proj_dual}, we report the $L^2$ errors and the associated orders of accuracy for $k=1,2,3,\, d=2,3.$ It is clear that the predicted order of accuracy is achieved.

\begin{table}[htp]
	
	\caption{$L^2$ errors and orders of accuracy for $L^2$ projection operator $\mathbf{\tilde P}_D$  of \eqref{ufunc} onto $\hat {\tilde {\bV}}_{N,D}^k$ when $d=2$ and $d=3$. 
		$N$ is the number of mesh levels,   $k$ is the polynomial order, $d$ is the dimension.  $L^2$ order is calculated with respect to  $h_N$. 
	}
	\centering
	\begin{tabular}{|c|c|c c|c c|c c|}
		\hline
		& & $L^2$ error & order& $L^2$ error & order & $L^2$  error & order\\
		\hline
		$N$& $h_N$&   \multicolumn{2}{c|}{$ k=1$}    &     \multicolumn{2}{c|}{$ k=2$}  &     \multicolumn{2}{c|}{$ k=3$}  \\
		\hline
		& &       \multicolumn{6}{c|}{$d=2$}     \\
		\hline
		3 &$1/8$  &  8.93E-04	&   --       &      9.14E-06	&	--     &	6.40E-08	&	--   \\
		4 &$1/16$ &	 2.61E-04	&	1.77     &	 	1.29E-06	&	2.82   &	4.45E-09	&	3.85   \\
		5 &$1/32$ &  7.34E-05	&	1.83     &	 	1.77E-07	&	2.87   &	3.01E-10	&	3.89   \\
		6 &$1/64$ &	 2.00E-05	&	1.88     &	 	2.37E-08	&	2.90   &	1.98E-11	&	3.93   \\
		7 &$1/128$&	 5.35E-06   &	1.90     &	 	3.11E-09	&	2.93   &	1.29E-12	&	3.94  \\

		\hline
		& &       \multicolumn{6}{c|}{$d=3$}     \\
		\hline
		3 &$1/8$  &	 6.19E-04	&   --	     &     4.93E-06	&	--	   &	3.18E-08	&	--	     \\
		4 &$1/16$ &  1.90E-04	&	1.70     &     7.45E-07	&	2.73   &	2.36E-09	&	3.75   \\
		5 &$1/32$ &  5.71E-05	&	1.73     &     1.10E-07	&	2.76   &	1.69E-10	&	3.80   \\
		6 &$1/64$ &  1.67E-05	&	1.77     &     1.58E-08	&	2.80   &	1.18E-11	&	3.84  \\
		7 &$1/128$&	 4.80E-06   &	1.80     &     2.24E-09	&	2.82   &	9.35E-13	&	3.66  \\
		\hline
		
	\end{tabular}
	\label{table:proj_dual}
\end{table}

\bigskip

With the aid of this space, the semi-discrete scheme can now be defined similarly as in \eqref{eq:DGformulation_p}-\eqref{eq:DGformulation_d} by using the space on the dual mesh as  $ \hat {\tilde {\bV}}_{N,D}^k,$ and replacing the numerical values on the boundary of the domain by corresponding functions in the Dirichlet boundary conditions.

We now comment on the implementation of this algorithm. As can be seen from Figures \ref{fig:nonp_basis_k02} and \ref{fig:nonp_basis_k12}, there are two types of basis functions in 1D for the dual space.
\begin{itemize}
\item Type 1 bases (for $ l\geq0$),  which are the shifted and truncated multiwavelet bases.
\item Type 2 bases (for $ l=0$), which are the Legendre polynomials of degree up to $k$ on $[1-\frac{h_N}{2}, 1]$.
\end{itemize}
Clearly, Type 1 bases are orthogonal to Type 2 bases, because their support do not overlap.  Type 2 bases are orthogonal to each other due to the definition of Legendre polynomials. However, Type 1 bases are no longer orthogonal to each other, due to the domain shift and truncation. However, only the left-most element on each level are changed. For other bases in that level, they will still retain orthogonality.   The bases on left-most element in all level are orthogonal to other bases, but not to each other, i.e., the bases defined on left-most element in different levels are not orthogonal. 
 This implies that although the mass matrix is not identity here, it will have  block structures and be sparse.

\begin{figure}[h!]
	\centering
	\subfigure[Primal mesh. Number of bases for $l=0, 1, 2$ are $1,1,2.$]{		\label{fig:nonp_basis_k01} \includegraphics[width=0.48\textwidth]{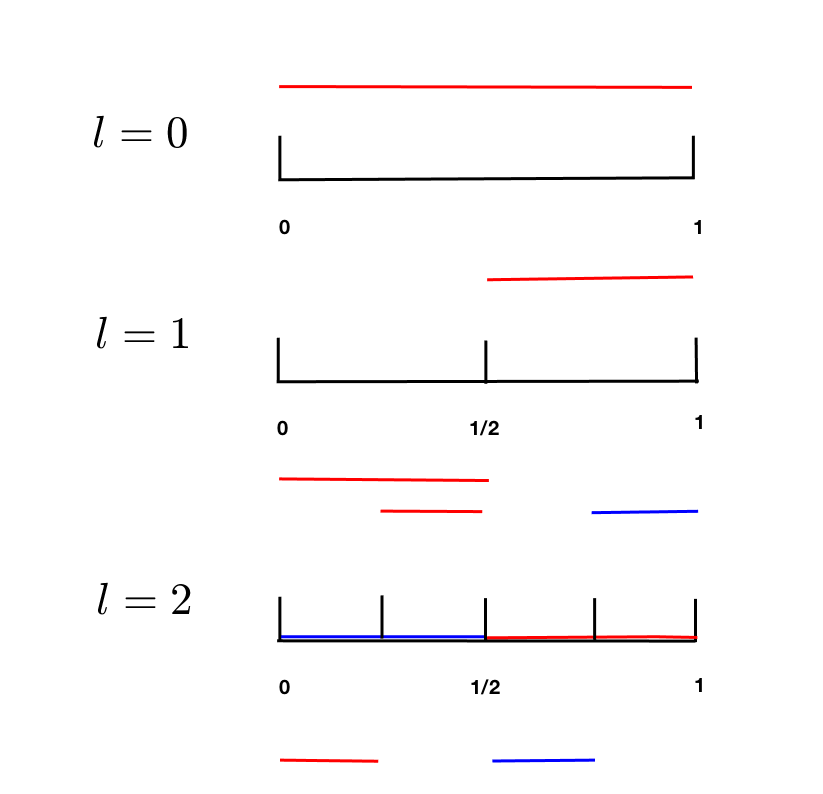}}
	\subfigure[ Dual mesh. Number of bases for $l=0, 1, 2$ are $2,1,2.$]{\label{fig:nonp_basis_k02} \includegraphics[width=0.48\textwidth]{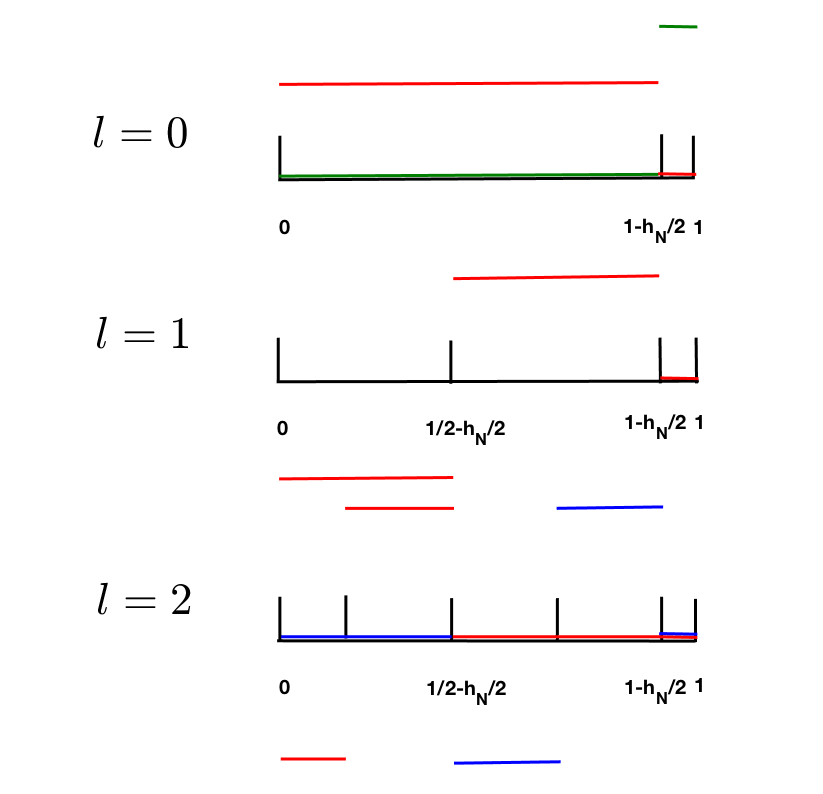}}
	\caption{Illustration of one-dimensional bases on different levels for $k = 0$: non-periodic problems. Different colors represent different bases.}
	\label{fig:nonp_basis_k0}
\end{figure}

\begin{figure}[h!]
	\centering
	\subfigure[Primal mesh. Number of bases for $l=0, 1, 2$ are $2,2,4.$]{\label{fig:nonp_basis_k11} \includegraphics[width=0.48\textwidth]{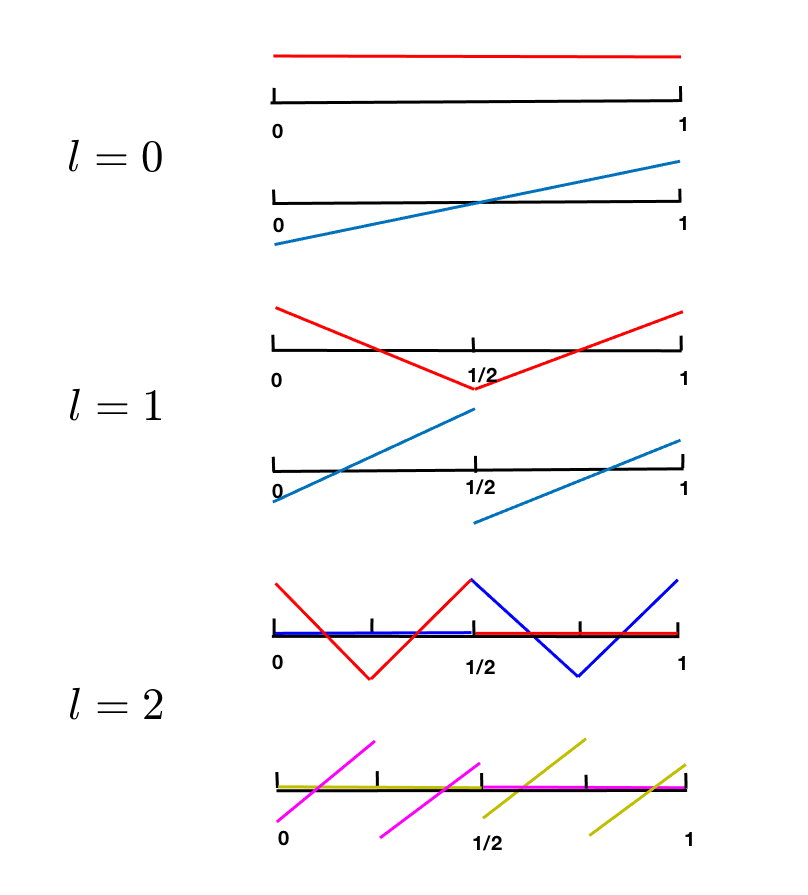}}
	\subfigure[ Dual mesh. Number of bases for $l=0, 1, 2$ are $4,2,4.$]{\label{fig:nonp_basis_k12} \includegraphics[width=0.48\textwidth]{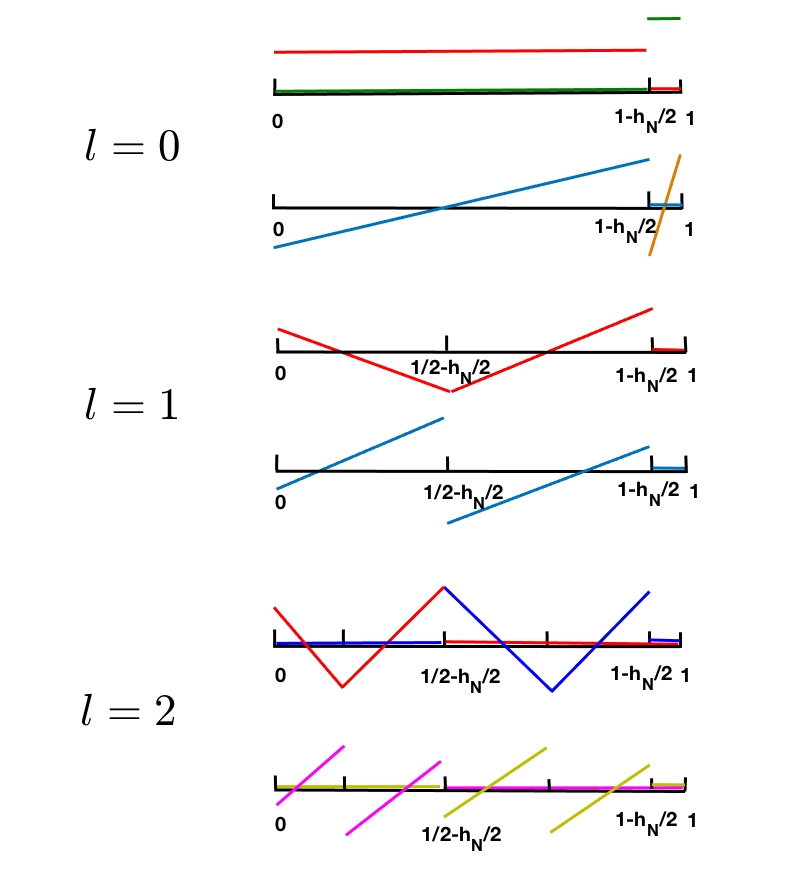}}
	\caption{Illustration of one-dimensional bases on different levels for $k = 1$: non-periodic problems.  Different colors represent different bases.}
	\label{fig:nonp_basis_k1}
\end{figure}

\section{Stability and convergence}
\label{sec:analysis}

In this section, we prove $L^2$ stability and error estimates for the sparse grid CDG scheme for the scalar equation.  We consider both periodic and non-periodic boundary conditions.
For periodic problems, \eqref{eq:model} reduces to
\begin{equation}
\label{eq:models}
\frac{\partial u}{\partial t} + \nabla \cdot (\mathbf{A} u ) =0,  \quad \bx\in\Omega,
\end{equation}
{\cb where ${\mathbf{A}} = (A_1(t, \bx), \cdots, A_d(t, \bx)),$ and $\|{\bf {A}} \|_{L^\infty(\Omega)} < \infty, \| \nabla \cdot {\bf {A}} \|_{L^\infty(\Omega)} < \infty.$ We assume $A_i \neq 0$ to avoid the discussion of different boundary conditions for degenerating coefficients. However, there is no difficulty to extend the proof below to degenerating case.} For non-periodic problems, the following inflow boundary conditions are prescribed,
$$u(t,\bx)|_{\partial \Omega_{x_i^{in}}} = g_i(t,\cdots,x_{i-1},x_{i+1},\cdots, x_d)$$ where 
$$\partial \Omega_{x_i^{in}} : =  \begin{cases} \{ \bx \in \Omega | x_i = 0\}, &  \text{if} \  A_i(t, \bx) >0, \\  \{ \bx \in \Omega | x_i = 1\}, &  \text{if} \  A_i <0. \end{cases}$$ 
Correspondingly, we denote the outflow edges by
$$\partial \Omega_{x_i^{out}} : =  \begin{cases} \{ \bx \in \Omega | x_i = 1\}, &  \text{if} \  A_i(t, \bx) >0,  	\\ \{ \bx \in \Omega | x_i = 0\}, &  \text{if} \  A_i <0. \end{cases}$$

The scheme for periodic case  reduces to:  to find $u_h\in\hat{\bV}_{N,P}^k$ and $v_h\in\hat{\bV}_{N,D}^k$, such that  
\begin{align}  
\label{eq:DGformulation_ps}
\int_{\Omega}(u_h)_t\,\ph_h\,d\bx =& \frac{1}{\ta}\int_{\Omega} (v_h-u_h) \,\ph_h\,d\bx  + \int_{\Omega}   v_h \mathbf{A} \cdot\nabla \ph_h\,d\bx - \sum_{\substack{e \in \Gamma_{N,P}}}\int_{e}  v_h \mathbf{A} \cdot  [\ph_h]   \,ds, \\
\label{eq:DGformulation_ds}
\int_{\Omega}(v_h)_t\,\psi_h\,d\bx =& \frac{1}{\ta}\int_{\Omega} (u_h-v_h) \,\psi_h\,d\bx +  \int_{\Omega}   u_h  \mathbf{A}  \cdot\nabla \psi_h\,d\bx - \sum_{\substack{e \in \Gamma_{N,D}}}\int_{e}   u_h \mathbf{A}  \cdot [\psi_h]  \,ds,
\end{align}
for any $\ph_h \in \hat{\bV}_{N,P}^k$ and $\psi_h \in \hat{\bV}_{N,D}^k.$
For non-periodic problems, we require $v_h, \psi_h \in \hat {\tilde {\bV}}_{N,D}^k,$ and enforce $u_h|_{\partial \Omega_{x_i^{in}}} =v_h|_{\partial \Omega_{x_i^{in}}}= g_i$ on the boundary interface.

We can prove that the schemes retain similar stability properties as the standard CDG schemes.

\begin{Thm}[$L^2$ Stability]
\label{thm:stability}
With periodic boundary condition,
the numerical solutions $u_h$ and $v_h$ of the sparse grid CDG scheme \eqref{eq:DGformulation_ps}-\eqref{eq:DGformulation_ds} for the equation \eqref{eq:models}  satisfy the following $L^2$ stability condition
\beq
\cb \|u_h\|^2_{L^2(\Omega_{N,P})} + \|v_h\|^2_{L^2(\Omega_{N,D})} \lesssim \|u_h(0,\bx)\|^2_{L^2(\Omega_{N,P})} + \|v_h(0,\bx)\|^2_{L^2(\Omega_{N,D})}.
\eeq

For non-periodic boundary condition, the corresponding numerical solutions satisfy
\beq
\cb \|u_h\|^2_{L^2(\Omega_{N,P})} + \|v_h\|^2_{L^2(\Omega_{N,D})} \lesssim \|u_h(0,\bx)\|^2_{L^2(\Omega_{N,P})} + \|v_h(0,\bx)\|^2_{L^2(\Omega_{N,D})} + \int_0^{T} \sum_{i=1}^d \int_{\partial \Omega_{x_i^{in}}} |A_i| g_i^2 ds \, dt
\eeq
  if $\ta \lesssim \frac{ h_N}{\|{\bf A}\|_1}.$
\end{Thm}

\begin{proof}
For periodic boundary condition,
let $\ph_h = u_h$ in \eqref{eq:DGformulation_ps} and $\psi_h = v_h$ in  \eqref{eq:DGformulation_ds}, summing the two equalities up, we have
\begin{align*}
&\ot \frac{d}{dt} \int_\Omega ( (u_h)^2 + (v_h)^2) d \bx\\
 =& \frac{1}{\ta}\int_{\Omega} v_h \,u_h - u_h \, u_h + u_h v_h - v_h v_h\,d\bx  
+  \into {v}_h {\bf{A}} \cdot\nabla u_h\,d\bx - \sum_{\substack{e \in \Gamma_{N,P}}}\int_{e}   v_h {\bf{A}} \cdot [u_h] \,ds	\\
&+  \into   {u}_h {\bf{A}} \cdot\nabla v_h\,d\bx - \sum_{\substack{e \in \Gamma_{N,D}}}\int_{e}   u_h {\bf{A}} \cdot [v_h] \,ds	\\
=& -\frac{1}{\ta} \into (u_h - v_h)^2 d\bx + \into {\bf{A}} \cdot \nabla (u_h v_h) d\bx - \sum_{\substack{e \in \Gamma_{N,P}}}\int_{e}   v_h {\bf{A}} \cdot [u_h] \,ds - \sum_{\substack{e \in \Gamma_{N,D}}}\int_{e}   u_h {\bf{A}} \cdot [v_h] \,ds.
\end{align*}

{\cb Apply divergence theorem, and by periodicity, we have
\[
\into {\bf{A}} \cdot \nabla (u_h v_h) d\bx - \sum_{\substack{e \in \Gamma_{N,P}}}\int_{e} {\bf{A}} v_h \cdot [u_h] \,ds - \sum_{\substack{e \in \Gamma_{N,D}}}\int_{e} {\bf{A}} u_h \cdot [v_h] \,ds = - \into \nabla \cdot {\bf{A}} u_h v_h d\bx.
\]

By the simple inequality $a b \leq \ot ( a^2 + b^2)$,
\[
\ot \frac{d}{dt} \int_{\Omega} \big ( (u_h)^2 + (v_h)^2 \big ) d\bx \leq -\frac{1}{\tau_{\max}} \int_{\Omega} (u_h -v_h)^2 d\bx + \ot \| \nabla \cdot {\bf {A}} \|_{L^\infty(\Omega)} \into ( (u_h)^2 + (v_h)^2) d \bx.
\]
and the proof for the periodic case is complete by using Gronwall's inequality. }

For non-periodic boundary condition, 
we follow the same lines and plug in the corresponding boundary condition,  
\begin{align*}
&\ot \frac{d}{dt} \int_\Omega ( (u_h)^2 + (v_h)^2) d \bx 
\\
 =& -\frac{1}{\ta} \into (u_h - v_h)^2 d\bx  {\cb - \into \nabla \cdot {\bf{A}} u_h v_h d\bx } +  \int_{\partial \Omega} {\bf A} \cdot \bn u_h v_h ds  - \sum_{i=1}^d \left ( \int_{\partial \Omega_{x_i^{in}}} {\bf A} \cdot \bn g_i(u_h + v_h) ds + 2 \int_{\partial \Omega_{x_i^{out}}} {\bf A} \cdot \bn u_h v_h  ds \right ) 	
 \\
  =& -\frac{1}{\ta} \into (u_h - v_h)^2 d\bx  {\cb - \into \nabla \cdot {\bf{A}} u_h v_h d\bx }	 +   \sum_{i=1}^d \left (\int_{\partial \Omega_{x_i^{in}}} |{\bf A} \cdot \bn| (-u_hv_h + g_i(u_h + v_h)) ds - \int_{\partial \Omega_{x_i^{out}}} |{\bf A} \cdot \bn| u_h v_h  ds \right)  
  \\
  \leq& -\frac{1}{\ta} \into (u_h - v_h)^2 d\bx {\cb + \ot \| \nabla \cdot {\bf {A}} \|_{L^\infty(\Omega)} \into ( (u_h)^2 + (v_h)^2) d \bx} 
  	\\
	& +  \sum_{i=1}^d \left ( \int_{\partial \Omega_{x_i^{in}}} | {\bf A} \cdot \bn | (g_i^2 +\ot(u_h - v_h)^2)  ds + \int_{\partial \Omega_{x_i^{out}}} | {\bf A} \cdot \bn | (\ot(u_h - v_h)^2-\ot u_h^2-\ot v_h^2) ds \right )
  \\
    \leq& -\frac{1}{\ta} \into (u_h - v_h)^2 d\bx  {\cb + \ot \| \nabla \cdot {\bf {A}} \|_{L^\infty(\Omega)} \into ( (u_h)^2 + (v_h)^2) d \bx} 
    	\\
	& +  \sum_{i=1}^d \left ( \int_{\partial \Omega_{x_i^{in}}} | {\bf A} \cdot \bn | g_i^2 ds  + \int_{\partial \Omega_{x_i^{in}} \cup \partial \Omega_{x_i^{out}}} | {\bf A} \cdot \bn | \ot(u_h - v_h)^2 ds \right )
    \\
    = & -\frac{1}{\ta} \into (u_h - v_h)^2 d\bx {\cb + \ot \| \nabla \cdot {\bf {A}} \|_{L^\infty(\Omega)} \into ( (u_h)^2 + (v_h)^2) d \bx} +  \sum_{i=1}^d \left ( \int_{\partial \Omega_{x_i^{in}}} |A_i| g_i^2 ds  + \int_{\partial \Omega_{x_i^{in}} \cup \partial \Omega_{x_i^{out}}} |A_i | \ot(u_h - v_h)^2 ds \right ).
\end{align*}
by noticing ${\bf A} \cdot \bn |_{\partial \Omega_{x_i^{in}}} < 0$ and ${\bf A} \cdot \bn|_{\partial \Omega_{x_i^{out}}} > 0.$


Let $T^{i}_{N,D} := \{ T \in \Omega_{N,D}| T \cap \partial \Omega_{x_i} \neq \emptyset \}$ denote the cells on dual mesh adjacent to the boundary in the $i$-th direction. 
By inverse inequality, we have $ \|u_h - v_h\|^2_{L^2(\partial \Omega_{x_i})}  \lesssim     h_N^{-1}  \|u_h - v_h\|^2_{L^2(T^{i}_{N,D} )} \le  h_N^{-1}  \|u_h - v_h\|^2_{L^2(\Omega)}.$  {\cb Therefore, 
\[
\ot \frac{d}{dt} \int_\Omega ( (u_h)^2 + (v_h)^2) d \bx \leq  \ot \| \nabla \cdot {\bf {A}} \|_{L^\infty(\Omega)} \into ( (u_h)^2 + (v_h)^2) d \bx + \sum_{i=1}^d \int_{\partial \Omega_{x_i^{in}}} |A_i| g_i^2 ds, \quad  \text{if } \ta \lesssim \frac{ h_N}{\|{\bf A}\|_1}
\]
and the proof for the non-periodic case is complete by using Gronwall's inequality.}
\end{proof}

\bigskip
Now we are ready to prove  $L^2$ error estimate of the sparse grid CDG scheme.

\begin{Thm}[$L^2$ error estimate]
\label{thm:lterr}
Let $u$ be the exact solution to \eqref{eq:models} and $u_h, v_h$ be the numerical solution to the semidiscrete scheme \eqref{eq:DGformulation_ps} and \eqref{eq:DGformulation_ds} with initial discretization $u_h(0, \bx) = \bP_P u_0, v_h(0, \bx) = \bP_D u_0$ for periodic boundary condition or $u_h(0, \bx) = \bP_P u_0, v_h(0, \bx) = \tilde \bP_D u_0$ for non-periodic boundary condition. If $\ta \lesssim h_N,$ then for $k\geq 1$, $u_0 \in \mathcal{H}^{p+1}(\Omega), 1\leq q \leq \min \{p,k\}, N \geq 1, d \geq 2$, we have for all $t \geq 0$
\beq
\label{eq:lterr}
\|u -u_h\|_{L^2(\Omega_{N,P})} + \|u -v_h\|_{L^2(\Omega_{N,D})} \lesssim N^d 2^{-Nq} \abs{u}_{\mathcal{H}^{q+1}(\Omega)}.
\eeq
\end{Thm}

\begin{proof}
For periodic problems, we first introduce the standard notation of bilinear form
\begin{align*}
B(u_h, v_h; \ph_h, \psi_h)  = & \int_{\Omega}(u_h)_t\,\ph_h\,d\bx - \frac{1}{\ta}\int_{\Omega} (v_h-u_h) \,\ph_h\,d\bx  - \int_{\Omega} v_h {\bf{A}} \cdot\nabla \ph_h\,d\bx + \sum_{\substack{e \in \Gamma_P}}\int_{e} v_h  {\bf{A}} \cdot [\ph_h] \,ds	\\
 + & \int_{\Omega}(v_h)_t\,\psi_h\,d\bx - \frac{1}{\ta}\int_{\Omega} (u_h - v_h) \,\psi_h\,d\bx - \int_{\Omega} u_h {\bf{A}} \cdot\nabla \psi_h\,d\bx  + \sum_{\substack{e \in \Gamma_D}}\int_{e}  u_h   {\bf{A}}  \cdot [\psi_h] \,ds.
\end{align*}

By Galerkin orthogonality, we have the error equation
\beq
\label{eq:erreq1}
B(u-u_h, u-v_h; \ph_h, \psi_h) = 0, \quad \forall \ph_h \in \hat{\bV}_{N,P}^k, \psi_h \in \hat{\bV}_{N,D}^k.
\eeq

We take
$$
\begin{array}{cc}
\ph_h = \bP_P u - u_h ,	& \psi_h = \bP_D u - u_h ,	\\
\ph^e = \bP_P u - u ,	& \psi^e = \bP_D u - u ,	\\
\end{array}
$$
then the error equation \eqref{eq:erreq1} becomes
\beq
\label{eqn:erreq2}
B(\ph_h, \psi_h; \ph_h, \psi_h) = B(\ph^e, \psi^e; \ph_h, \psi_h).
\eeq

From  Theorem \ref{thm:stability}, we get
\beq
\label{eq:erreqlhs}
\ot \frac{d}{dt} \intom \big ( \ph_h^2 + \psi_h^2 \big ) d\bx \leq B(\ph^e, \psi^e; \ph_h, \psi_h) {\cb + \ot \| \nabla \cdot {\bf {A}} \|_{L^\infty(\Omega)} \into ( \ph_h^2 + \psi_h^2) d \bx}.
\eeq


We write the bilinear form on the right-hand side as a sum of three terms
\beq
\label{eq:erreqrhs}
B(\ph^e, \psi^e; \ph_h, \psi_h) = B^1 + B^2 + B^3,
\eeq
where
\begin{align*}
& B^1 = \int_{\Omega}(\ph^e)_t\,\ph_h\,d\bx - \frac{1}{\ta}\int_{\Omega} (\psi^e-\ph^e) \,\ph_h\,d\bx + \int_{\Omega}(\psi^e)_t\,\psi_h\,d\bx - \frac{1}{\ta}\int_{\Omega} (\ph^e - \psi^e) \,\psi_h\,d\bx ,								\\
& B^2 = - \int_{\Omega} \psi^e {\bf{A}} \cdot\nabla \ph_h\,d\bx - \int_{\Omega} \ph^e {\bf{A}} \cdot\nabla \psi_h\,d\bx ,	\\
& B^3 = \sum_{\substack{e \in \Gamma_{N,P}}}\int_{e}  \psi^e {\bf{A}} \cdot [\ph_h] \,ds + \sum_{\substack{e \in \Gamma_{N,D}}}\int_{e}   \ph^e {\bf{A}} \cdot [\psi_h]  \,ds.
\end{align*}

By Cauchy-Schwartz inequality, Lemma \ref{lem:ltproj} and $\ta \lesssim h_N$, we have
\beq
\label{eq:b1}
B^1 \lesssim \intom (\ph_h^2 + \psi_h^2 )d\bx  + N^{2d} 2^{-2Nq}\abs{u}^2_{\mathcal{H}^{q+1}(\Omega)}. 	
\eeq

To estimate $B^2, B^3$, we use the following inverse inequalities  $\forall w_h \in \hat{\bV}_{N,G}^k,$ for $ G = P,D$,
$$
| w_h |_{\mathcal{H}^1(\Omega_{N,G})} \lesssim h_N^{-1} \|w_h \|_{L^2(\Omega_{N,G})}, \quad \| w_h\|_{\Gamma_{N,G}} \lesssim h_N^{-\ot} \|w_h \|_{L^2(\Omega_{N,G})}
$$
and trace inequality,
$$
\| \phi \|_{L^2(\partial T)}^2 \lesssim {h_N}^{-1} \| \phi \|_{L^2(T)}^2 + h_N | \phi |_{H^1(T)}, \quad \forall \phi \in H^1(T), T \in \Omega_{N,G}.
$$

Then we have
\beq
\label{eq:b2}
B^2 \lesssim \intom (\ph_h^2 + \psi_h^2 )d\bx + N^{2d} 2^{-2Nq}\abs{u}^2_{\mathcal{H}^{q+1}(\Omega)}
\eeq
and
\beq
\label{eq:b3}
B^3 \lesssim \intom (\ph_h^2 + \psi_h^2 )d\bx + N^{2d} 2^{-2Nq}\abs{u}^2_{\mathcal{H}^{q+1}(\Omega)}.
\eeq

Combining \eqref{eq:b1}, \eqref{eq:b2}, \eqref{eq:b3} with \eqref{eq:erreqlhs}, we obtain
$$
\frac{d}{dt} \intom \big ( \ph_h^2 + \psi_h^2 \big ) d\bx \lesssim \intom (\ph_h^2 + \psi_h^2 )d\bx + N^{2d} 2^{-2Nq} \abs{u}^2_{\mathcal{H}^{q+1}(\Omega)}.
$$

Together with the estimates for initial discretization   and by Gronwall's inequality, the proof is complete. For non-periodic problems, the argument is very similar as long as  the stability result holds. The proof is omitted for brevity.
\end{proof}

This theorem proves $L^2$ error of the scheme is $O(N^d 2^{-Nk})$ or $O(\abs{\log h_N}^d h_N^k)$ when the exact solution has enough smoothness in the mixed derivative norms. 

\section{Numerical results}
\label{sec:numerical}

In this section, we present several numerical tests to validate the performance of the proposed sparse grid CDG schemes. Unless otherwise stated, we use the third-order TVD-RK temporal discretization \cite{Shu_1988_JCP_NonOscill}   and choose the time step   $\Delta t = \frac{c}{\displaystyle\sum_{i=1}^d \frac{c_i}{h_N}},$ 
with $c=0.1$ for $k=1,\,2$, where $c_i$ is the maximum wave propagation speed in $x_i$-direction. To guarantee that the spatial error dominates for $k=3$, we take $\Delta t = O(h_N^{4/3}).$   $\ta$ is taken as $\frac{1}{2k+1}h_N$ which is always smaller than the maximum time step allowed based on the CFL number in Table \ref{table:linear_cfl}. For periodic problems, we only provide   $L^2$ errors on the primal mesh, because the results on the dual mesh are similar. For non-periodic problems, the $L^2$ errors are the $L^2$ average of the errors on the primal and dual meshes.

\subsection{Scalar case}
In this subsection, we consider the scalar case, i.e. $m=1.$

\begin{exa}[Linear advection with constant coefficients]
	\label{ex:linear}
	We consider  
	\begin{equation}
	\label{eq:linear_adv}
	\left\{\begin{array}{l} \displaystyle u_t + \sum_{i=1}^d u_{x_i} = 0,\quad \bx\in[0,1]^d,\\[2mm]
	\displaystyle u(0,\bx) = \sin\left(2\pi\sum_{i=1}^d x_i\right),
	\end{array}\right.
	\end{equation}
	with periodic or Dirichlet boundary conditions on the inflow edges corresponding to the given exact solution.
\end{exa}

The exact solution   is a smooth function, 
$$u(t,\bx) =  \sin\left(2\pi\left(\sum_{i=1}^d x_i - d\,t \right)\right).$$
In the simulation, we compute the numerical solutions up to two periods in time, meaning that we let final time $T=1$ for $d=2$, $T=2/3$ for $d=3$, and $T=0.5$ for $d=4$. 

We first test the scheme with periodic boundary condition.
In Table \ref{table:linear_d2}, we report the $L^2$ errors and  orders of accuracy for $k=1, 2, 3$ and up to dimension four. As for accuracy, we observe about half order reduction from the optimal $(k+1)$-th order for high-dimensional computations ($d=4$). The order is slightly better for lower dimensions. The convergence order is similar to the performance of the sparse grid DG scheme in \cite{guo2016sparse}. In Figure \ref{fig:linear}, we plot the time evolution of the error of $L^2$ norm of numerical solutions $u_h$ and $v_h$, which is given by
$$\int_{\Omega} \big ( (u_h(t, \bx))^2 + (v_h(t, \bx))^2 \big ) d\bx - \int_{\Omega} \big ( (u_h(0, \bx))^2 + (v_h(0, \bx))^2 \big ) d\bx$$
for two-dimensional case for $t=0$ to $t=100$. From Theorem \ref{thm:stability}, such errors are proportional to the difference between $u_h$ and $v_h.$ We can clearly see that the higher order accurate scheme performs way better in conservation of $L^2$ norm due to its higher order accuracy.

\begin{table}[htp]
	
	\caption{$L^2$ errors and orders of accuracy for Example \ref{ex:linear} at $T=1$ when $d=2$, $T=2/3$ when $d=3$, and $T=0.5$ when $d=4$. 
		$N$ denotes mesh level, $h_N$ is the size of the smallest mesh in each direction, $k$ is the polynomial order, $d$ is the dimension.  $L^2$ order is calculated with respect to  $h_N$. 
	}
	\centering
	\begin{tabular}{|c|c|c c|c c|c c|}
		\hline
		& & $L^2$ error & order& $L^2$ error & order & $L^2$  error & order\\
		\hline
		$N$& $h_N$&   \multicolumn{2}{c|}{$ k=1$}    &     \multicolumn{2}{c|}{$ k=2$}  &     \multicolumn{2}{c|}{$ k=3$}  \\
		\hline
		& &       \multicolumn{6}{c|}{$d=2$}     \\
		\hline
		3 &$1/8$  &  3.14E-01	&   --       &     1.20E-02	&	--     &	5.84E-04	&	--   \\
		4 &$1/16$ &	 6.99E-02	&	2.17     &     2.23E-03	&	2.43   &	8.50E-05	&	2.78   \\
		5 &$1/32$ &  1.34E-02	&	2.38     &	   4.87E-04	&	2.20   &	3.84E-06	&	4.47   \\
		6 &$1/64$ &	 3.43E-03	&	1.97     &	   5.97E-05	&	3.03   &	3.89E-07	&	3.30   \\
		7 &$1/128$&	 9.21E-04   &	1.90     &	   9.33E-06	&	2.68   &	1.80E-08	&	4.43  \\

		\hline
		& &       \multicolumn{6}{c|}{$d=3$}     \\
		\hline
		3 &$1/8$  &	 6.77E-01	&   --	     &     5.27E-02	&	--	   &	2.13E-03	&	--	     \\
		4 &$1/16$ &  3.56E-01	&	0.93     &     1.10E-02	&	2.26   &	2.62E-04	&	3.02   \\
		5 &$1/32$ &  1.05E-01	&	1.76     &     1.82E-03	&	2.60   &	2.85E-05	&	3.20   \\
		6 &$1/64$ &  2.54E-02	&	2.05     &     5.22E-04	&	1.80   &	2.01E-06	&	3.83  \\
		7 &$1/128$&	 7.45E-03   &	1.77     &     6.89E-05	&	2.92   &	2.01E-07	&	3.32  \\
		\hline
		
		& &       \multicolumn{6}{c|}{$d=4$}     \\
		\hline
		3 &$1/8$  &	 7.13E-01	&   --	     &     1.26E-01	&	--	   &	4.41E-03	&	--	     \\
		4 &$1/16$ &  6.48E-01	&	0.14     &     3.39E-02	&	1.89   &	7.56E-04	&	2.54   \\
		5 &$1/32$ &  3.80E-01	&	0.77     &     6.91E-03	&	2.29   &	9.82E-05	&	2.94   \\
		6 &$1/64$ &  1.37E-01	&	1.47     &     1.39E-03	&	2.31   &	9.44E-06	&	3.38  \\
		7 &$1/128$&	 3.81E-02   &	1.85     &     3.56E-04	&	1.97   &	8.16E-07	&	3.53  \\
		\hline
		
	\end{tabular}
	\label{table:linear_d2}
\end{table}

	\begin{figure}[htp]
	\begin{center}
		\subfigure[k=1]{\includegraphics[width=.32\textwidth]{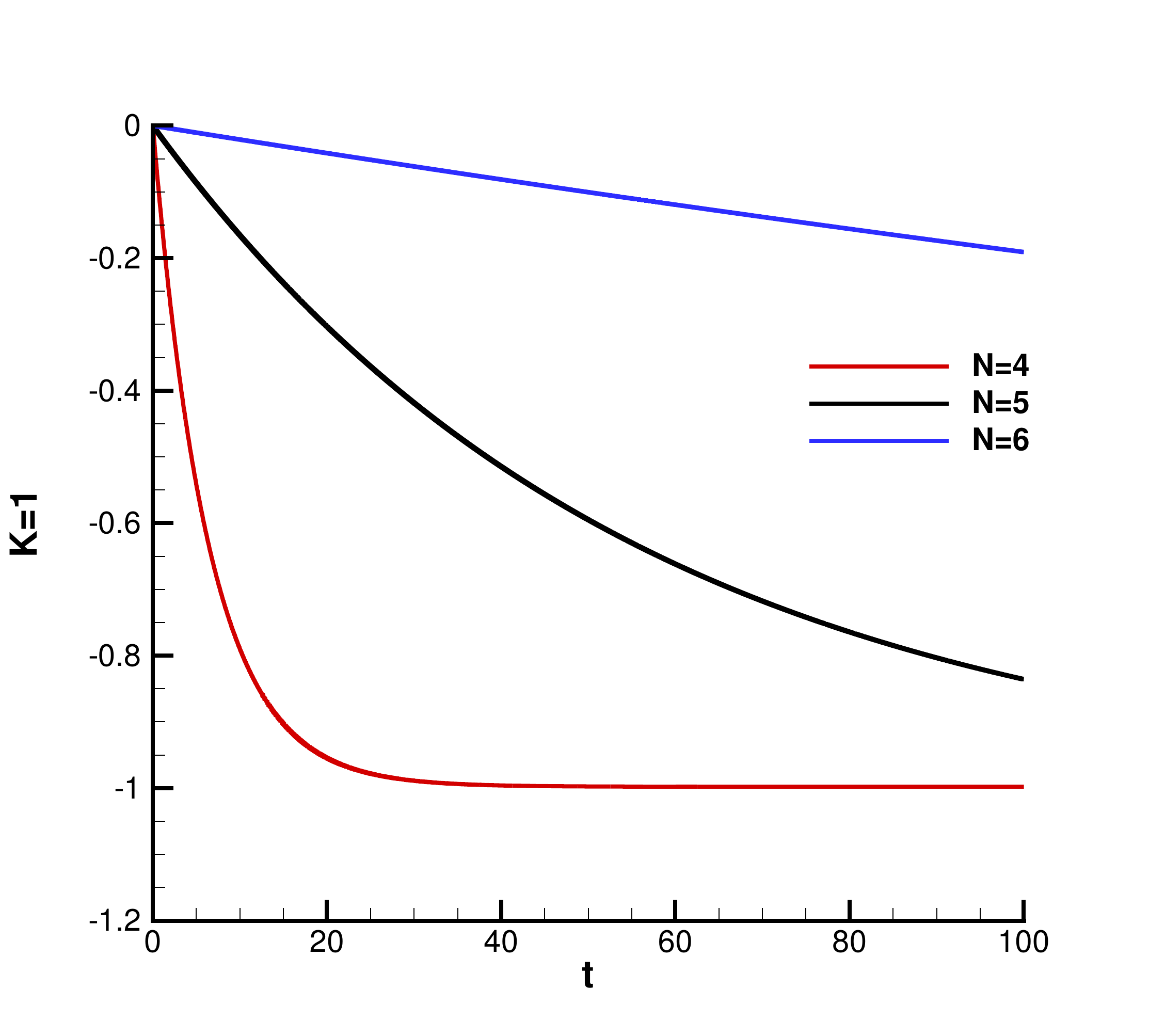}}
		\subfigure[k=2]{\includegraphics[width=.32\textwidth]{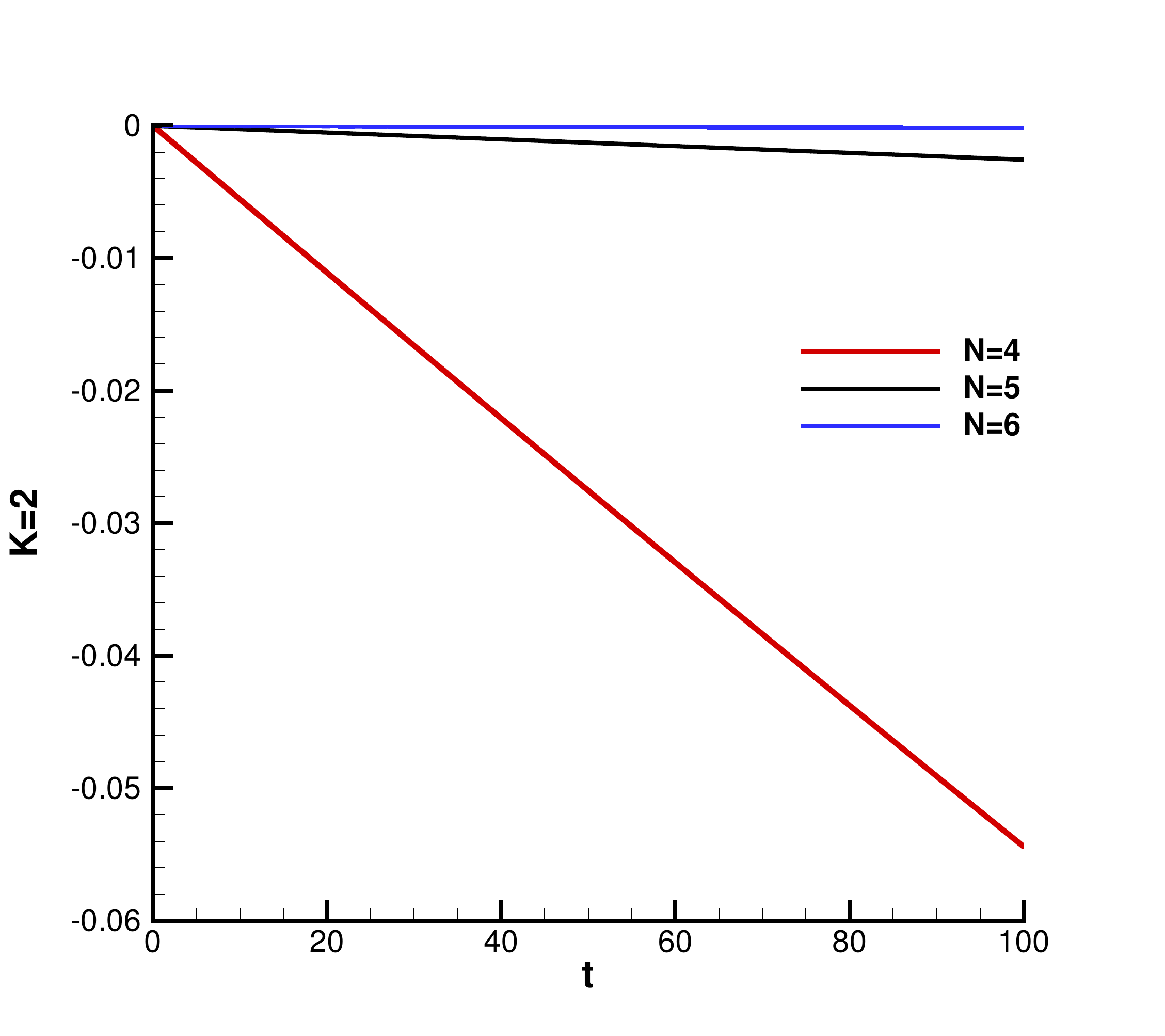}}
		\subfigure[k=3]{\includegraphics[width=.32\textwidth]{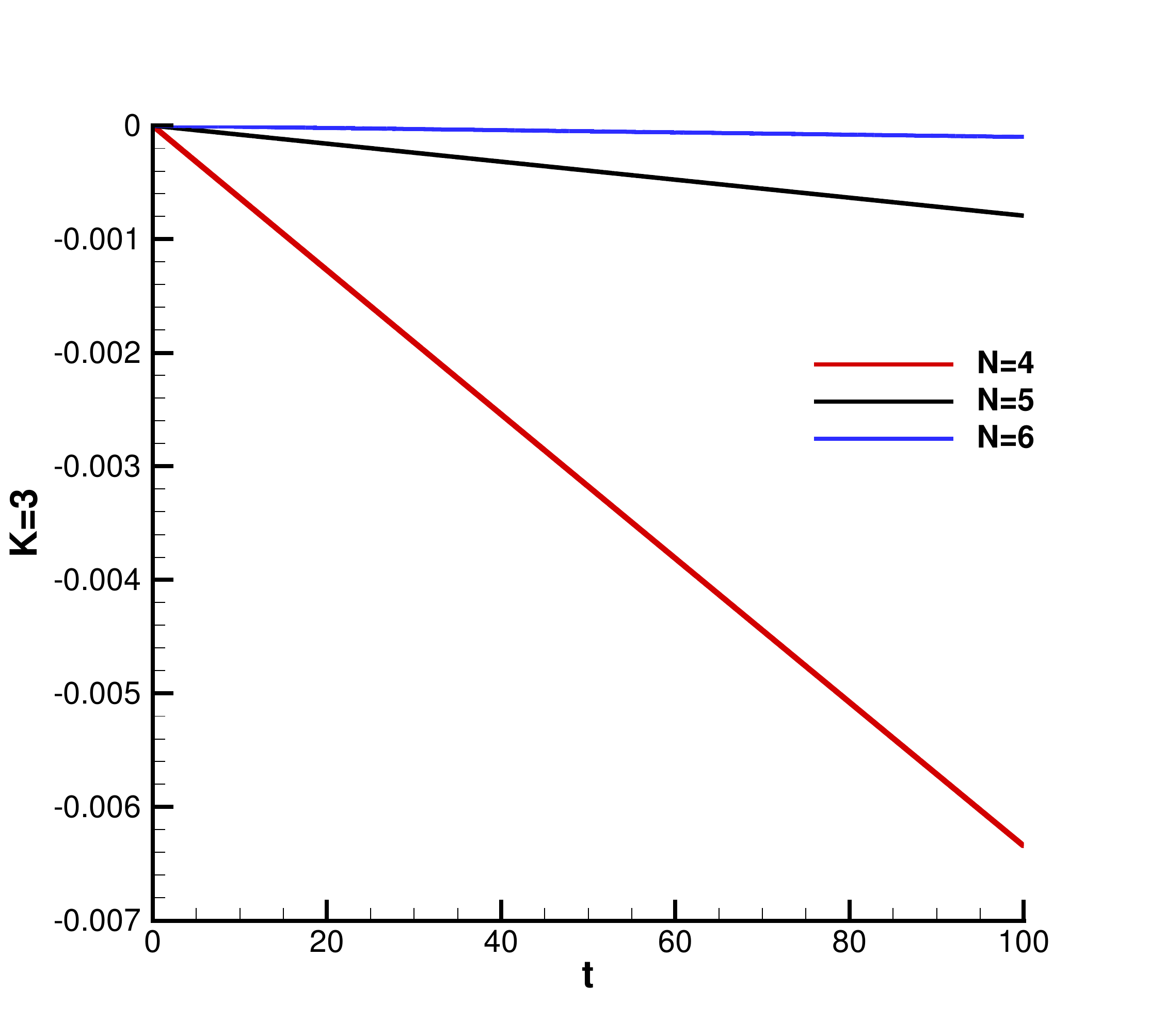}}
	\end{center}
	\caption{Example \ref{ex:linear}. The time evolution of the error of $L^2$ norm of numerical solutions $u_h$ and $v_h$ of the sparse grid CDG method with $d=2.$ (a) k=1, (b) k=2, (c) k=3. $N=4,5,6$. }
	\label{fig:linear}
    \end{figure}

Then, we test the scheme with Dirichlet boundary condition prescribed at the inflow edge according to the exact solution. The results are listed in Table \ref{table:linear_nonp}. The accuracy order is similar to the periodic case.

\begin{table}[htp]
	
	\caption{$L^2$ errors and orders of accuracy for Example \ref{ex:linear} with Dirichlet boundary condition on the inflow edges at $T=1$ when $d=2$ and $T=2/3$ when $d=3$. 
		$N$ denotes mesh level, $h_N$ is the size of the smallest mesh on the primal mesh in each direction,  $k$ is the polynomial order, $d$ is the dimension. $L^2$ order is calculated with respect to  $h_N$. 
	}
	\centering
	\begin{tabular}{|c|c|c c|c c|c c|}
		\hline
		& & $L^2$ error & order& $L^2$ error & order & $L^2$  error & order\\
		\hline
		$N$& $h_N$&   \multicolumn{2}{c|}{$ k=1$}    &     \multicolumn{2}{c|}{$ k=2$}  &     \multicolumn{2}{c|}{$ k=3$}  \\
		\hline
		& &       \multicolumn{6}{c|}{$d=2$}     \\
		\hline
		3 &$1/8$  &  2.66E-01	&   --       &      1.66E-02	&	--     &	8.21E-04	&	--   \\
		4 &$1/16$ &	 7.47E-02	&	1.83     &	 	3.33E-03	&	2.32   &	8.80E-05	&	3.22   \\
		5 &$1/32$ &  1.94E-02	&	1.95     &	 	5.97E-04	&	2.48   &	4.79E-06	&	4.20   \\
		6 &$1/64$ &	 5.44E-03	&	1.83     &	 	8.60E-05	&	2.80   &	4.50E-07	&	3.41   \\
		7 &$1/128$&	 1.49E-03    &	1.87     &	 	1.35E-05	&	2.67   &	2.20E-08	&	4.35  \\

		\hline
		& &       \multicolumn{6}{c|}{$d=3$}     \\
		\hline
		3 &$1/8$  &	 6.15E-01	&   --	     &     5.34E-02	&	--	   &	2.67E-03	&	--	     \\
		4 &$1/16$ &  2.86E-01	&	1.10     &     1.40E-02	&	1.93   &	2.87E-04	&	3.22   \\
		5 &$1/32$ &  1.14E-01	&	1.33     &     2.57E-03	&	2.45   &	3.21E-05	&	3.16   \\
		6 &$1/64$ &  3.23E-02	&	1.82     &     5.82E-04	&	2.14   &	2.60E-06	&	3.63  \\
		7 &$1/128$&	 1.03E-02   &	1.65     &     9.81E-05	&	2.57   &	2.86E-07	&	3.18  \\
		\hline
		
	\end{tabular}
	\label{table:linear_nonp}
\end{table}

{\cb
Finally, we use this example to compare the performance of the DG, CDG, sparse grid DG and sparse grid CDG methods.  We use the following non-separable initial condition
\begin{equation}
\label{initial}
u(0, \mathbf{x})=\exp\left(\sin\left(2\pi\sum_{i=1}^d x_i\right) \right),\quad \mathbf{x}\in[0,1]^d,
\end{equation}
where $d=2.$  When $k=1, 2, 3$, Runge-Kutta methods of order $\nu = 2,3,4$, respectively, are used for time discretization. We take the time step according to the 
CFL numbers listed in Table \ref{table:linear_cfl}. We plot the comparison of the methods measuring $L^2$ errors vs. CPU times in Figure \ref{fig:cpu_2d}.  The computations in this example are implemented by an OpenMP code  using computational resources from the Institute for Cyber-Enabled Research in  Michigan State University. We can see that the sparse grid CDG method outperforms the CDG method, and the sparse grid DG method outperforms the DG method particularly when the mesh level $N$ is more refined.
When the mesh level increases from $N$ to $N+1$, the CPU cost for sparse grid method grows with the rate of about 4 to 5, while the factor is about 8 to 10 for  full grid calculations, respectively, for this 2D case. This shows the advantage of the sparse grid approach. When comparing the sparse grid CDG method with the sparse grid DG method, it seems that for this example, the sparse grid DG method is more efficient. It will be interesting to compare the results for fully nonlinear problems in higher dimensions, for which the CDG method is more advantageous, and this is currently under investigation.
}
 
 \begin{figure}[htp]
	\begin{center}
		\subfigure[k=1]{\includegraphics[width=.32\textwidth]{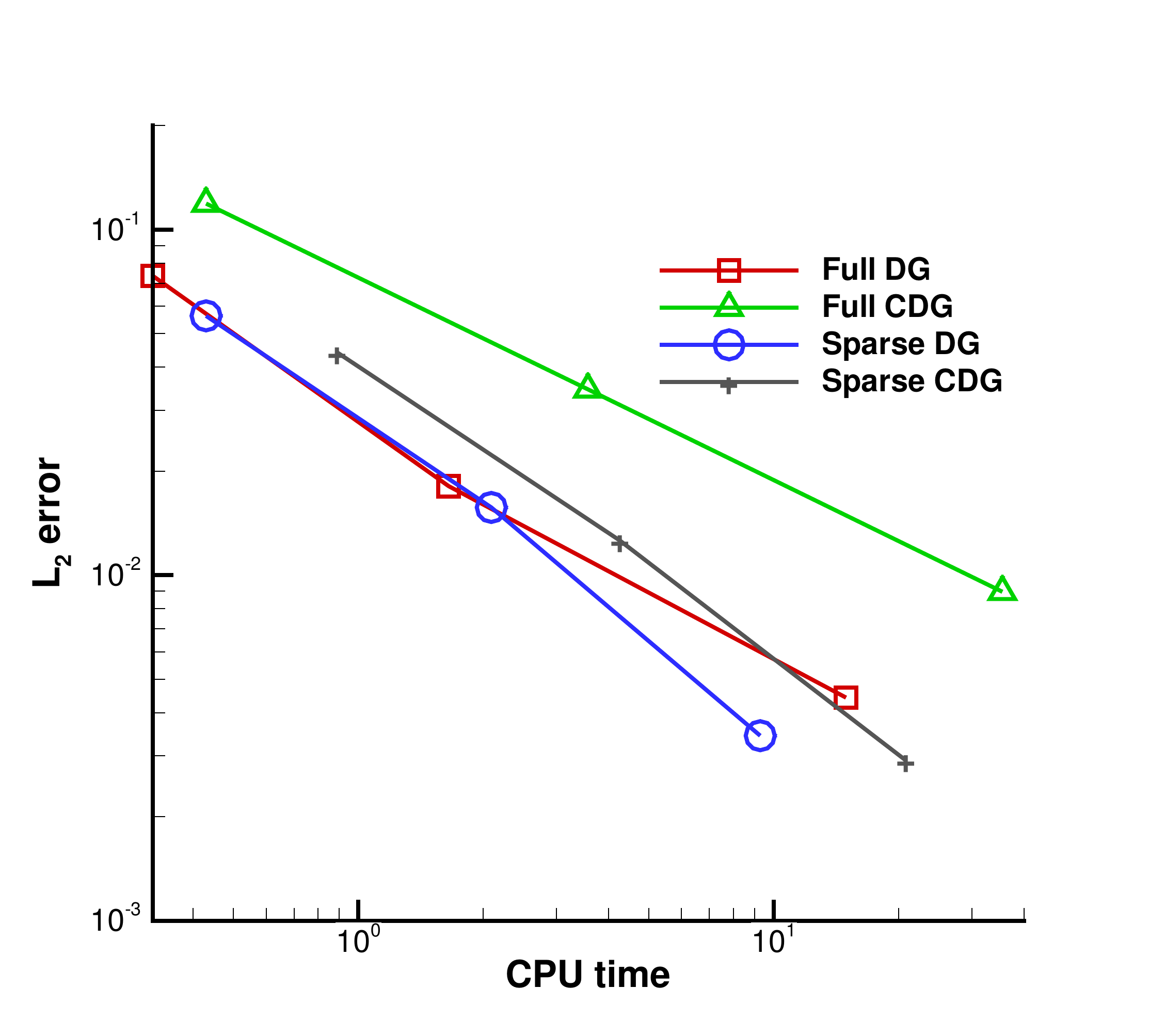}}
		\subfigure[k=2]{\includegraphics[width=.32\textwidth]{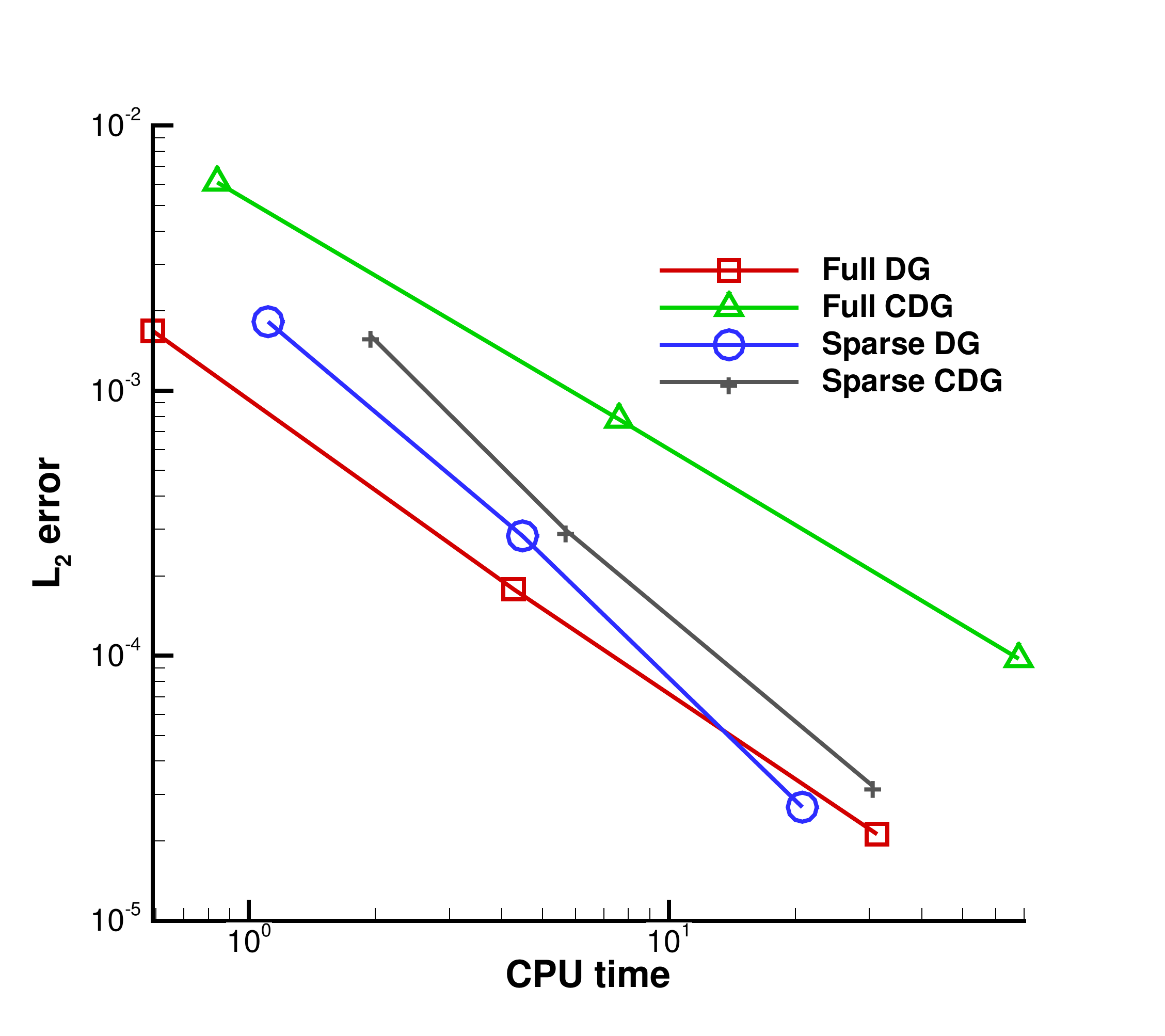}}
		\subfigure[k=3]{\includegraphics[width=.32\textwidth]{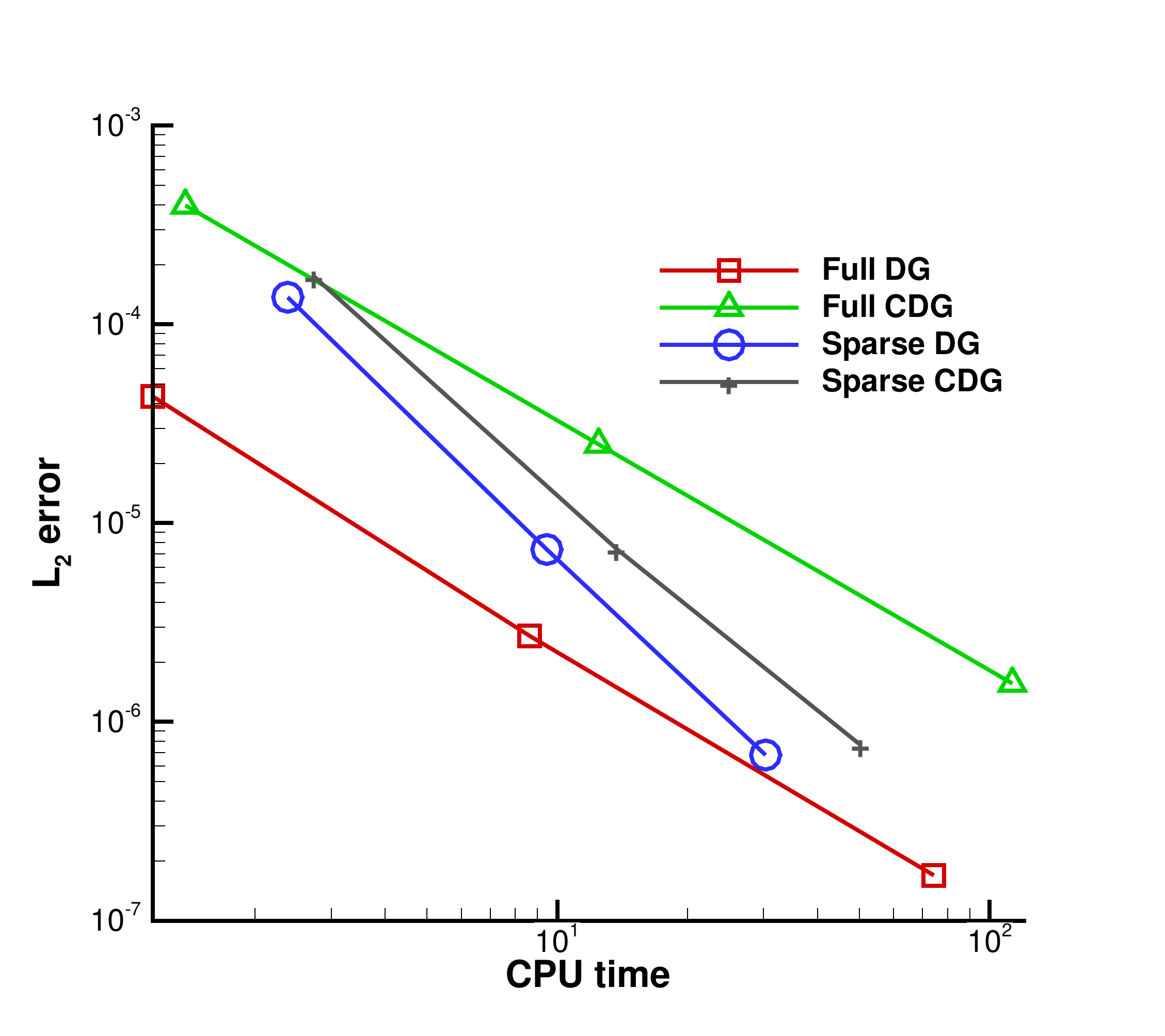}}
	\end{center}
	\caption{ $L^2$ errors and associated CPU times of DG, CDG, sparse grid DG and sparse grid  CDG methods for Example 4.1 with initial condition \eqref{initial}  at $T=1$ for d=2. (a) k=1, (b) k=2, (c) k=3. }
	\label{fig:cpu_2d}
\end{figure}

\begin{exa}[Solid body rotation]
	\label{ex:rotation}
	We consider solid-body-rotation problems, which are in the form of \eqref{eq:model0} with
	$$A_1(t, \bx) = -x_2+\frac12,\, A_2(t, \bx) = x_1-\frac12,\quad d=2,$$
	$$A_1(t, \bx) = -\frac{\sqrt{2}}{2}\left(x_2-\frac12\right), \,
	A_2(t, \bx) = \frac{\sqrt{2}}{2}\left(x_1-\frac12\right)+ \frac{\sqrt{2}}{2}\left(x_3-\frac12\right), \,
	A_3(t, \bx) = -\frac{\sqrt{2}}{2}\left(x_2-\frac12\right),\quad d=3,$$
	subject to periodic boundary conditions.
\end{exa}

Such benchmark tests are commonly used in the literature to assess performance of transport schemes. Here,  the initial profile traverses along circular trajectories centered at $(1/2,1/2)$ for $d=2$ and about the axis $\{x_1=x_3\}\cap \{x_2=1/2\}$ for $d=3$ without deformation, and it goes back to the initial state after $2\pi$ in time.
The initial conditions are set to be the following smooth cosine bells (with $C^5$ smoothness),
\begin{equation}\label{eq:cosine} u(0,\bx)=\left\{\begin{array}{ll}b^{d-1}\cos^6\left(\frac{\pi r}{2b}\right),& \text{if}\quad r\leq b,\\
0,&\text{otherwise},
\end{array}\right.\end{equation}
where $b=0.23$ when $d=2$ and $b=0.45$ when $d=3$, and $r=|\bx-\bx_c|$ denotes the distance between $\bx$ and the center of the cosine bell with $\bx_c=(0.75,0.5)$ for $d=2$ and  $\bx_c=(0.5,0.55,0.5)$ for $d=3$.  

In Table \ref{table:rot}, we summarize the convergence study of the numerical solutions computed by the sparse CDG method, including the $L^2$ errors and   orders of accuracy. For this variable coefficients equation, we observe at least $k$-th order convergence for all cases. The order is slightly lower than the corresponding ones in Example \ref{ex:linear}.

%

\begin{table}[htp]
	
	\caption{$L^2$ errors and orders of accuracy for Example \ref{ex:rotation} at $T=2\pi$. $N$ denotes mesh level, $h_N$ is the size of the smallest mesh in each direction, $k$ is the polynomial order, $d$ is the dimension.  $L^2$ order is calculated with respect to  $h_N$.
	}
	\centering
	\begin{tabular}{|c|c|c c|c c|c c|}
		\hline
		& & $L^2$ error & order& $L^2$ error & order  & $L^2$  error & order\\
		\hline
		$N$& $h_N$&   \multicolumn{2}{c|}{$ k=1$}    &     \multicolumn{2}{c|}{$ k=2$}  &    \multicolumn{2}{c|}{$ k=3$}  \\
		\hline
		& &       \multicolumn{6}{c|}{$d=2$}     \\
		\hline
		5 &$1/32$ &    1.53E-02  &  --       &   5.81E-03    &  --     &    1.34E-03   & --  \\
		6 &$1/64$ &    1.02E-02  &	0.58     &	 1.50E-03    &  1.95   &    9.64E-05   & 3.80  \\
		7 &$1/128$&    4.66E-03  &	1.13     &	 1.46E-04	 &	3.36   &    1.16E-05   & 3.05  \\
		8 &$1/256$&    1.42E-03	 &  1.71     &   2.34E-05	 &	2.64   &    1.10E-06   & 3.40  \\

		\hline
		& &       \multicolumn{6}{c|}{$d=3$}     \\
		\hline
		5 &$1/32$ 	&  4.83E-03   & --       	&  6.25E-04  	&  --  	&    7.35E-05 	&	 --  \\
		6 &$1/64$ 	&  1.87E-03   & 1.37   	&  1.20E-04  	& 2.38 	&    9.18E-06 	& 	3.00 \\
		7 &$1/128$	&  7.46E-04   & 1.33     	&  3.39E-05  	& 1.82 	&  1.36E-06 	& 	2.75  \\
		8 &$1/256$	&  2.55E-04   & 1.55     	&  8.11E-06  	& 2.06 	& 	1.94E-07	& 	2.81  \\
		\hline
		
	\end{tabular}
	\label{table:rot}
\end{table}

\begin{exa}[Deformational flow]
	\label{ex:deformational}
	We consider the two-dimensional deformational flow with velocity field
	$$A_1(t, \bx) =  \sin^2(\pi x_1)\sin(2\pi x_2)g(t),\, A_2(t, \bx) = -\sin^2(\pi x_2)\sin(2\pi x_1)g(t),$$
	where $g(t)=\cos(\pi t/T)$ with $T=1.5,$ with periodic boundary condition. 
\end{exa}

We still adopt the cosine bell \eqref{eq:cosine} as the initial condition for this test, but with $\bx_c=(0.65,0.5)$ and $b=0.35$.
Note that the deformational test is more challenging than the solid body rotation due to the space and time dependent flow field. In particular, along the direction of the flow, the cosine bell deforms into a crescent shape at $t=T/2$ , then goes back to its initial state at $t=T$ as the flow reverses. In the simulations, we compute the solution up to $t=T$. 
The convergence study is summarized in Table \ref{table:deform_d2}. {Similar orders are observed compared with Example \ref{ex:rotation}. In Figure \ref{fig:defor}, we plot the contour plots of the numerical solutions on the primal mesh at $t=T/2$ when the shape of the bell is greatly deformed, and $t=T$ when the solution is recovered into its initial state. It is observed that the sparse CDG scheme with higher degree $k$ can better resolve the highly deformed solution structure. 

	\begin{table}[htp]
		\caption{$L^2$ errors and orders of accuracy for Example \ref{ex:deformational} at $T=1.5$. $N$ denotes mesh level, $h_N$ is the size of the smallest mesh in each direction, $k$ is the polynomial order, $d$ is the dimension.  $L^2$ order is calculated with respect to  $h_N$. $d=2$. 
		}
		\centering
		\begin{tabular}{|c|c | c c| c c| c c|}
			\hline
			$N$ & $h_N$&$L^2$ error & order & $L^2$ error & order & $L^2$ error & order\\
			\hline
			
			& &        \multicolumn{2}{c|}{$ k=1$} &  \multicolumn{2}{c|}{$ k=2$} &  \multicolumn{2}{c|}{$ k=3$}  \\
			\hline
			5	& 1/32 &	 1.73E-02	&	-- &	  4.37E-03	&	-- & 1.14E-03	&	-- \\
			6	& 1/64 &	 8.06E-03	&   1.10 &  1.17E-03	&	1.90 &     2.44E-04	& 2.22 \\
			7	& 1/128 &    3.29E-03	&	1.29 & 2.04E-04	&	2.52 &   2.05E-05	&	3.57\\
			8	& 1/256 &	 1.08E-03   &	1.61 &  2.78E-05	&	2.88 &    2.75E-06	&	2.90\\
			\hline

		\end{tabular}
		\label{table:deform_d2}
	\end{table}

	\begin{figure}[htp]
		\begin{center}
			\subfigure[]{\includegraphics[width=.42\textwidth]{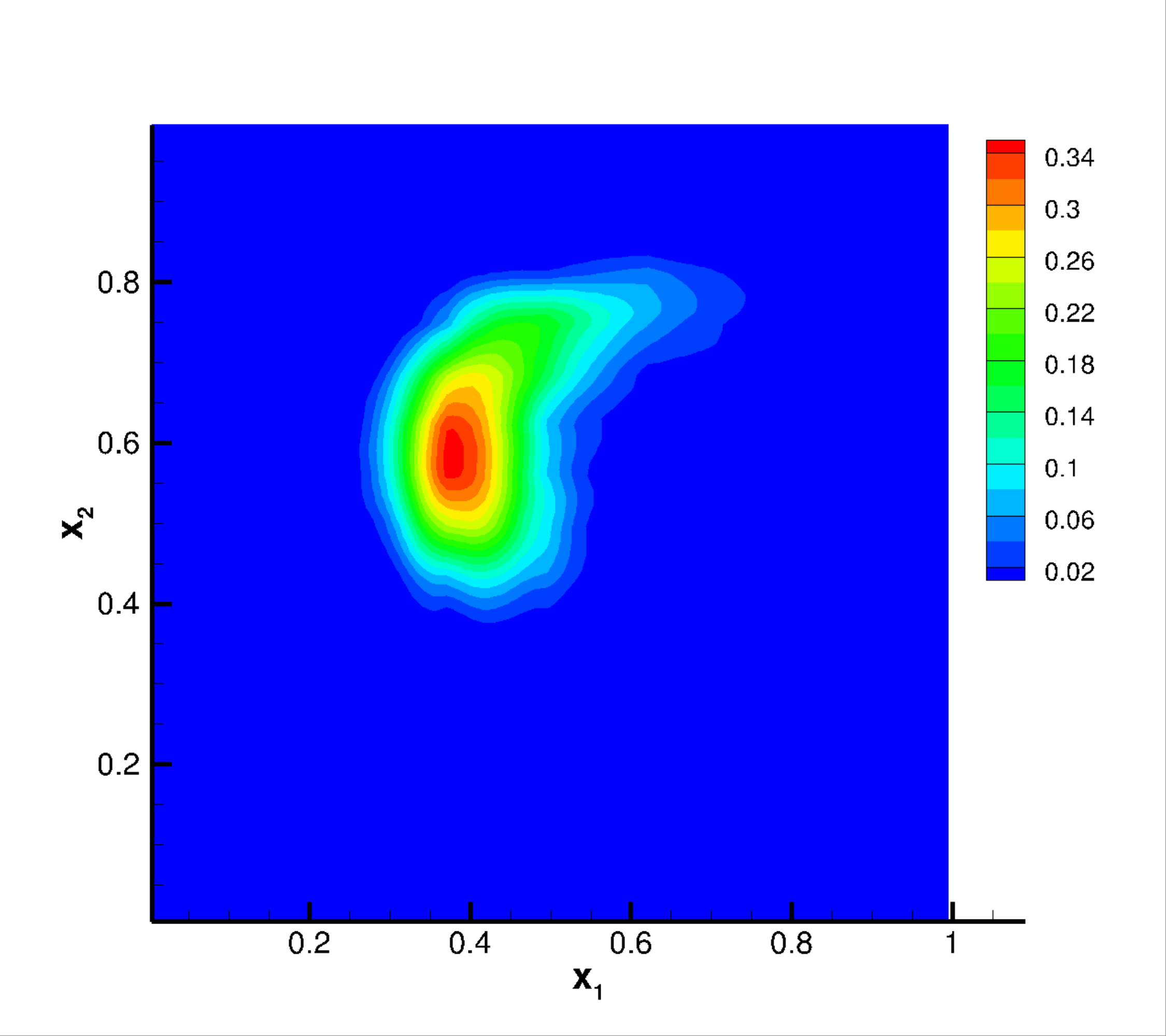}}
			\subfigure[]{\includegraphics[width=.42\textwidth]{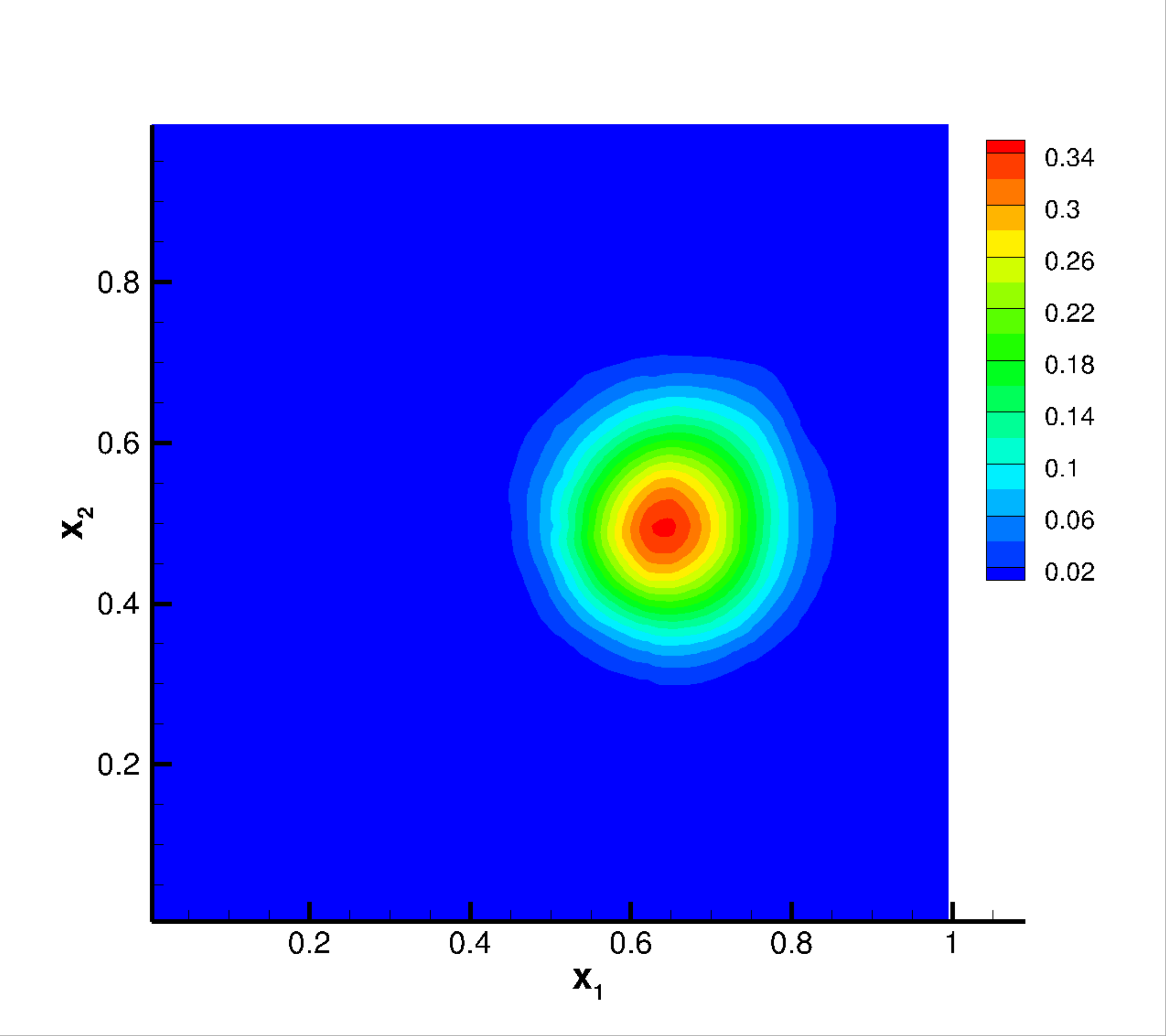}}\\
			\subfigure[]{\includegraphics[width=.42\textwidth]{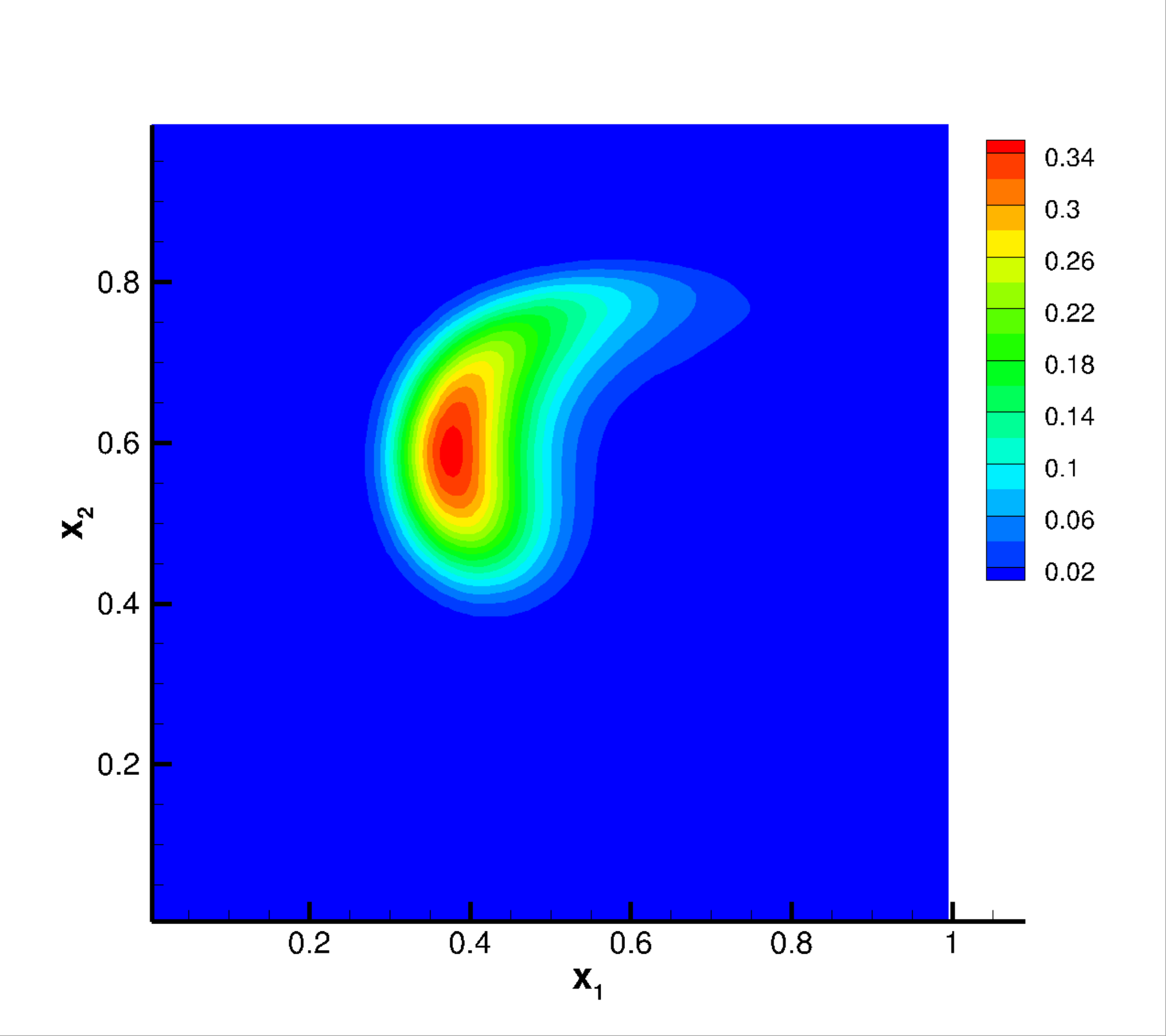}}
			\subfigure[]{\includegraphics[width=.42\textwidth]{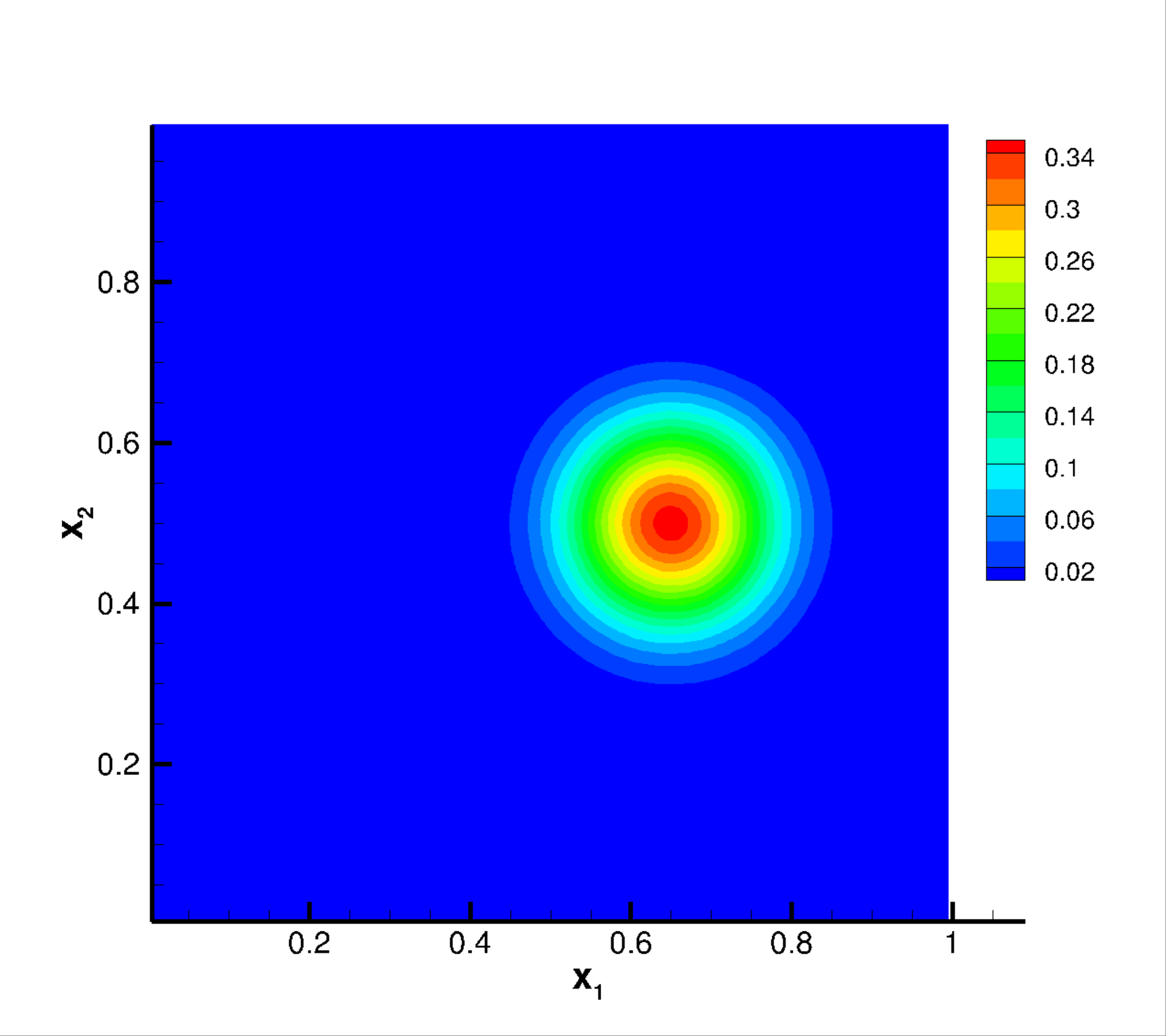}}\\
			\subfigure[]{\includegraphics[width=.42\textwidth]{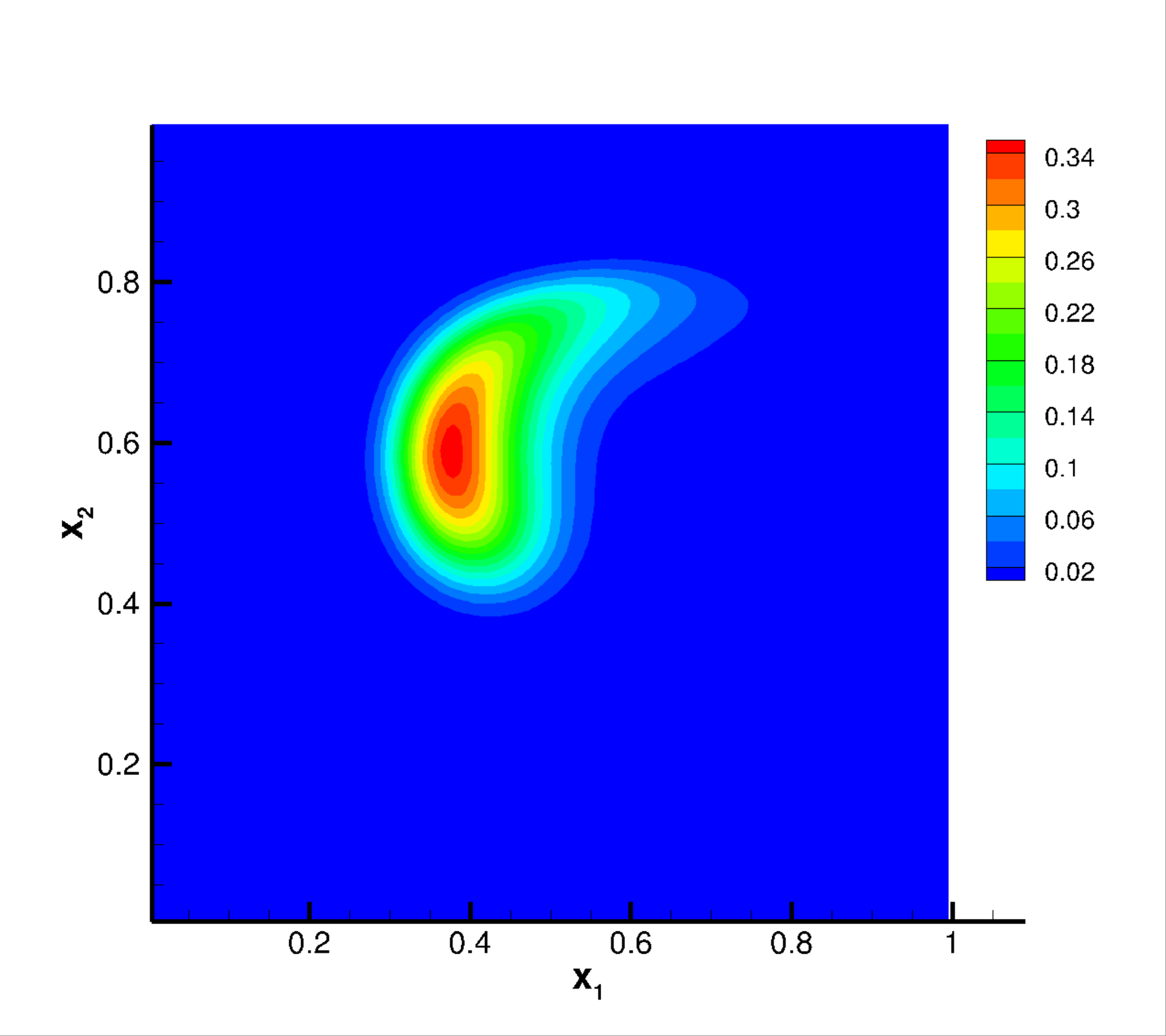}}
			\subfigure[]{\includegraphics[width=.42\textwidth]{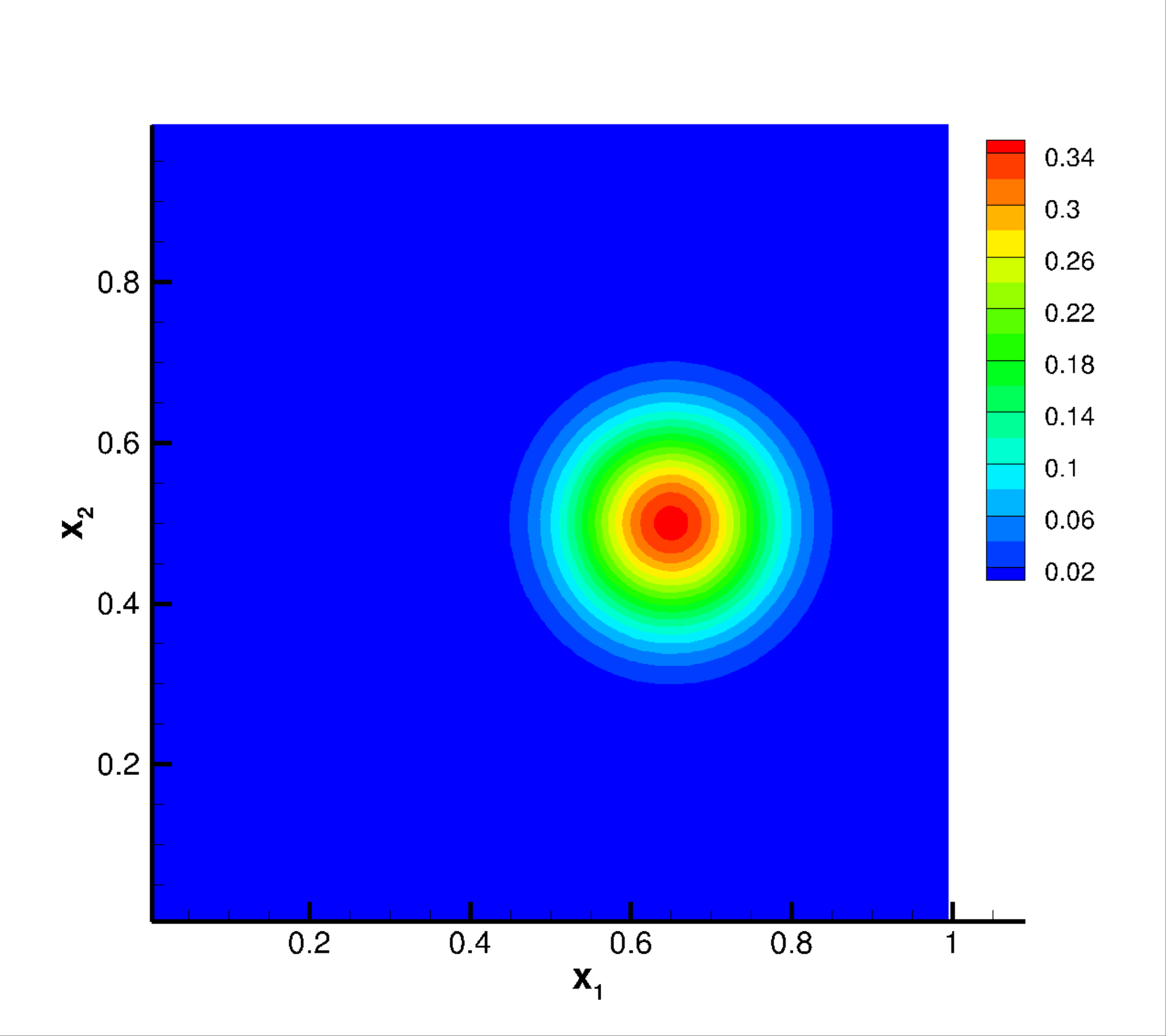}}
		\end{center}
		\caption{Example \ref{ex:deformational}. Deformational flow test. The contour plots of the numerical solutions on primal mesh  at $t=T/2$ (a, c, e) and $t=T$ (b, d, f). $k=1$ (a, b), $k=2$ (c, d), and $k=3$ (e, f). $N=7$. }
		\label{fig:defor}
	\end{figure}

\subsection{System case}
In this subsection, we consider system case, which means $m>1$ in  equation \eqref{eq:model0} or \eqref{eq:model}.

\begin{exa}[Acoustic wave equation with constant wave speed]
	\label{ex:system_Acoustic}
	We consider  
	\begin{equation}
	\left\{ \begin{aligned} 
	&u_t=\nabla \cdot \bv, \quad \bx\in[0,1]^2,\\
	&\bv_t=\nabla u, \\
	& u(0,\bx) = u_0(\bx), \quad \bv(0,\bx) = \bv_0(\bx).
	\end{aligned} \right. 
	\end{equation}	
%
%
	with periodic boundary conditions.
The initial conditions $u_0(\bx)$ and $\bv_0(\bx)$ are chosen according to the following two types of exact solutions: the standing wave
	$$
		\begin{bmatrix}
	u(t,\bx) 	\\ 
	v_1(t,\bx) 	\\
	v_2(t,\bx)
	\end{bmatrix}
	 = 
	\begin{bmatrix}
	-\sqrt 2  \sin(2\sqrt 2\pi t) \sin(2\pi x_1) \sin(2\pi x_2)	\\
	 \cos(2\sqrt 2\pi t) \cos(2\pi x_1) \sin(2\pi x_2)	\\
	 \cos(2\sqrt 2\pi t) \sin(2\pi x_1) \cos(2\pi x_2)	\\
	\end{bmatrix},
	$$
      and the traveling wave
	$$
		\begin{bmatrix}
	u(t,\bx) 	\\ 
	v_1(t,\bx) 	\\
	v_2(t,\bx)
	\end{bmatrix}
	 = 
	\begin{bmatrix}
	\sqrt 2 \sin(2\sqrt 2\pi t + 2\pi x_1)\cos (2\pi x_2)	\\
	 \sin(2\sqrt 2\pi t+ 2\pi x_1) \cos(2\pi x_2)	\\
	\cos(2\sqrt 2\pi t +2\pi x_1) \sin(2\pi x_2)	\\
	\end{bmatrix}.
	$$

%
\end{exa}

We compute the solution until $T=1.$ Similar to the scalar case, we present the $L^2$ errors and   orders of accuracy for $\bu(t,\bx) = \bmat u(t,\bx),	& v_1(t,\bx),	& v_2(t,\bx) \emat^T$ in Table \ref{Acoustic_wave}. 
From the table, we still observe at least $(k+1/2)$-th order for the solution.  
	
	\begin{table}[htp]
		
		\caption{$L^2$ errors and orders of accuracy for Example \ref{ex:system_Acoustic} at $T=1$. $N$ denotes mesh level, $h_N$ is the size of the smallest mesh in each direction, $k$ is the polynomial order, $d$ is the dimension.  $L^2$ order is calculated with respect to  $h_N$. $d=2$. 
		}
		\centering
		\begin{tabular}{|c|c|c c|c c|c c|}
			\hline
			& & $L^2$ error & order& $L^2$ error & order & $L^2$  error & order\\
			\hline
			$N$& $h_N$&   \multicolumn{2}{c|}{$ k=1$}    &     \multicolumn{2}{c|}{$ k=2$}  &     \multicolumn{2}{c|}{$ k=3$}  \\
			\hline
			& &       \multicolumn{6}{c|}{standing wave}     \\
			\hline
			3 &$1/8$  &  3.56E-01	&   --       &     1.05E-02	&	--     &	5.37E-04	&	--   \\
			4 &$1/16$ &	 7.93E-02	&	2.17     &     1.84E-03	&	2.51   &	4.31E-05	&	3.64   \\
			5 &$1/32$ &  1.50E-02	&	2.40     &	   3.18E-04	&	2.53   &	3.39E-06	&	3.67   \\
			6 &$1/64$ &	 3.72E-03	&	2.01     &	   4.95E-05	&	2.68   &	2.77E-07	&	3.61   \\
			7 &$1/128$&	 1.01E-03   &	1.88     &	   7.60E-06	&	2.70   &	2.03E-08	&	3.77  \\

			\hline
			& &       \multicolumn{6}{c|}{traveling wave}     \\
			\hline
			3 &$1/8$  &	 3.97E-01	&   --	     &     1.85E-02	&	--	   &	7.75E-04	&	--	     \\
			4 &$1/16$ &  8.58E-02	&	2.21     &     3.36E-03	&	2.46   &	6.76E-05	&	3.52   \\
			5 &$1/32$ &  1.97E-02	&	2.12     &     6.07E-04	&	2.47   &	5.68E-06	&	3.57   \\
			6 &$1/64$ &  5.36E-03	&	1.88     &     9.66E-05	&	2.65   &	4.44E-07	&	3.68  \\
			7 &$1/128$&	 1.50E-03    &	1.84     &     1.45E-05	&	2.74   &	3.39E-08	&	3.71  \\
			\hline
			
		\end{tabular}
		\label{Acoustic_wave}
	\end{table}
	
%
	
\begin{exa}[Two-dimensional homogeneous isotropic elastic wave \cite{kaser2006arbitrary}]
\label{ex:system_elastic_2d}
The 2D elastic wave equation in homogeneous and isotropic medium in velocity-stress formulation without external source,   is a linear hyperbolic system of the form
\beq
\label{eqn:elastic2d}
\bu _t + A_1 \bu_{x_1} + A_2 \bu_{x_2} =0,
\eeq
where $\bu = \bmat \sigma_{xx},	& \sigma_{yy},	& \sigma_{xy} ,	& v, 	& w \emat^T$, $\sigma_{xx}, \sigma_{yy}$ represents the normal stress and $\sigma_{xy}$ represents the shear stress and $v, w$ are the velocity in $x$ and $y$ directions. 
$$
A_1 = -\bmat		0	& 	0	& 	0	& 	\lambda + 2\mu 	& 	0	\\
			0	& 	0	& 	0	& 	\lambda 			& 	0	\\
			0	& 	0	& 	0	& 	0				& 	\mu	\\
	\frac{1}{\rho}	& 	0	& 	0	& 	0				& 	0	\\
			0	& 	0	& 	\frac{1}{\rho}	& 	0		& 	0	\\
	\emat 	,	\quad 
A_2 = -\bmat		0	& 	0	& 	0	& 	0		& \lambda 		\\
			0	& 	0	& 	0	& 	0 		& 	\lambda + 2\mu 	\\
			0	& 	0	& 	0	& 	\mu				& 	0	\\
			0	& 	0	& 	\frac{1}{\rho}	& 	0				& 	0	\\
			0	& 	\frac{1}{\rho}	& 	0	& 	0		& 	0	\\
	\emat ,
$$
where $\lambda$ and $\mu$ are the \textit{Lam\'e constants} and $\rho$ is the mass density of material. Eigenvalues of $A_1$ and $A_2$ are $-c_p, -c_s, 0, c_s, c_p$, which give us the wave speed $c_p = \sqrt{\frac{\lambda + 2\mu}{\rho}}$ and $c_s = \sqrt{\frac{\mu}{\rho}}$ for P-wave and S-wave respectively. 
We consider the homogeneous material parameters $\lambda = 2, \mu = 1, \rho = 1$, then $c_p = 2, c_s = 1$. 
On domain $\Omega = [0,1]^2$, we take the solutions consisting of a plane P-wave traveling along diagonal direction $\mathbf n = (\frac{\sqrt 2}{2}, \frac{\sqrt 2}{2})$ and a plane S-wave traveling in the opposite direction, i.e.,
$$
\bu (t,\bx) = \mathbf R_s e^{\sin(\mathbf{k} \cdot \bx + kc_s t)} + \mathbf R_p e^{\sin(\mathbf{k} \cdot \bx -kc_p t)},
$$
where $\mathbf R_s = [-\mu, \mu, 0, -\frac{\sqrt 2}{2}c_s, \frac{\sqrt 2}{2}c_s]^T, \mathbf R_p = [\lambda + \mu, \lambda + \mu, \mu, -\frac{\sqrt 2}{2}c_p, -\frac{\sqrt 2}{2}c_p]^T$ and $\mathbf{k} = k \mathbf n, k = {2\sqrt 2\pi}$. Periodic boundary condition is applied and the initial condition is chosen as $u(0,\bx)$.
\end{exa}

We compute the solution until $T=1.$ The $L^2$ errors and  orders of accuracy for $\bu(t,\bx)$ are shown in Table \ref{elastic_2d}.  We observe that the convergence order is close to $k+1$.

\begin{table}[htp]
	
	\caption{$L^2$ errors and orders of accuracy for Example \ref{ex:system_elastic_2d} at $T=1$. $N$ denotes mesh level, $h_N$ is the size of the smallest mesh in each direction, $k$ is the polynomial order, dimension $d=2$.  $L^2$ order is calculated with respect to  $h_N$. 
	}
	\centering
	\begin{tabular}{|c|c|c c|c c|c c|}		
		\hline
		& & $L^2$ error & order& $L^2$ error & order & $L^2$  error & order\\
		\hline
		$N$& $h_N$&   \multicolumn{2}{c|}{$ k=1$}    &     \multicolumn{2}{c|}{$ k=2$}  &     \multicolumn{2}{c|}{$ k=3$}  \\
		\hline
		4 &$1/16$ &  1.09E+00	&	--       &     2.72E-01	&	--     &	5.71E-02	&	--   \\
		5 &$1/32$ &  7.47E-01	&	0.55     &     6.48E-02	&	2.07   &	6.19E-03	&	3.21   \\
		6 &$1/64$ &  2.41E-01	&	1.63     &     9.65E-03	&	2.75   &	4.77E-04	&	3.70  \\
		7 &$1/128$&	 7.14E-02   &	1.76     &     1.12E-03	&	3.11   &	2.55E-05	&	4.23  \\
		\hline
		
	\end{tabular}
	\label{elastic_2d}
\end{table}

\begin{exa}[Three-dimensional isotropic elastic wave \cite{dumbser2006arbitrary}]
	\label{ex:system_elastic_3d}
We extend the previous example to 3D and obtain the following linear hyperbolic system
\beq
\label{eqn:elastic3d}
\bu _t + A_1 \bu_{x_1} + A_2 \bu_{x_2} + A_3 \bu_{x_3}=0,
\eeq
where $\bu = \bmat \sigma_{xx},	& \sigma_{yy},	&\sigma_{zz}, 	& \sigma_{xy}, 	& \sigma_{yz},	& \sigma_{xz},	& u, 	& v, 	& w \emat^T$, $\sigma$ is the stress tensor and $u, v, w$ are the velocities in each spatial direction.
$$
A_1 = -\bmat	0	&	0	&	0	&	0	&	0	&	0	&	\lambda + 2\mu	&	0	&	0	\\
		0	&	0	&	0	&	0	&	0	&	0	&	\lambda			&	0	&	0	\\
		0	&	0	&	0	&	0	&	0	&	0	&	\lambda			&	0	&	0	\\
		0	&	0	&	0	&	0	&	0	&	0	&	0				&	\mu	&	0	\\
		0	&	0	&	0	&	0	&	0	&	0	&	0				&	0	&	0	\\
		0	&	0	&	0	&	0	&	0	&	0	&	0				&	0	&	\mu	\\
\frac{1}{\rho}	&	0	&	0	&	0	&	0	&	0	&	0				&	0	&	0	\\
0	&	0	&	0	&	\frac{1}{\rho}	&	0	&	0	&	0				&	0	&	0	\\
0	&	0	&	0	&	0	&	0	&	\frac{1}{\rho}	&	0				&	0	&	0	\\
	\emat ,	\,
A_2 = -\bmat	0	&	0	&	0	&	0	&	0	&	0	&	0	&	\lambda			&	0	\\
		0	&	0	&	0	&	0	&	0	&	0	&	0	&	\lambda+2\mu	&	0	\\
		0	&	0	&	0	&	0	&	0	&	0	&	0	&	\lambda			&	0	\\
		0	&	0	&	0	&	0	&	0	&	0	&	\mu	&	0				&	0	\\
		0	&	0	&	0	&	0	&	0	&	0	&	0	&	0				&	\mu	\\
		0	&	0	&	0	&	0	&	0	&	0	&	0	&	0				&	0	\\
0	&	0	&	0	&	\frac{1}{\rho}	&	0	&	0	&	0	&	0				&	0	\\
0	&	\frac{1}{\rho}	&	0	&	0	&	0	&	0	&	0	&	0				&	0	\\
0	&	0	&	0	&	0	&	\frac{1}{\rho}	&	0	&	0	&	0				&	0	\\
	\emat ,	
$$
$$
A_3 = -\bmat	0	&	0	&	0	&	0	&	0	&	0	&	0	&	0	&	\lambda	\\
		0	&	0	&	0	&	0	&	0	&	0	&	0	&	0	&	\lambda	\\
		0	&	0	&	0	&	0	&	0	&	0	&	0	&	0	&	\lambda+2\mu	\\
		0	&	0	&	0	&	0	&	0	&	0	&	0	&	0	&	0	\\
		0	&	0	&	0	&	0	&	0	&	0	&	0	&	\mu	&	0	\\
		0	&	0	&	0	&	0	&	0	&	0	&	\mu	&	0	&	0	\\
		0	&	0	&	0	&	0	&	0	&	\frac{1}{\rho}	&	0	&	0	&	0	\\
		0	&	0	&	0	&	0	&	\frac{1}{\rho}	&	0	&	0	&	0	&	0	\\
		0	&	0	&	\frac{1}{\rho}	&	0	&	0	&	0	&	0	&	0	&	0	\\
	\emat ,
$$
where $\lambda, \mu$ and $\rho$ take the same values as the previous example. Hence, we have the same values for $c_p$ and $c_s$. Eigenvalues of $A_1,A_2$ and $A_3$ are $-c_p, -c_s, -c_s, 0, 0, 0, c_s, c_s, c_p$, which describe the wave speed for P-wave and S-wave (with different polarizations).
On domain $\Omega = [0,1]^3$, we take the solutions consisting of a plane S-wave traveling along diagonal direction $\mathbf n = (-\frac{1}{\sqrt 3}, -\frac{1}{\sqrt 3}, -\frac{1}{\sqrt 3})$ and a plane P-wave traveling in the opposite direction, i.e.,
$$
\bu (t,\bx) = \mathbf R_s \sin(\mathbf{k} \cdot \bx - k c_s t) + \mathbf R_p \sin( \mathbf{k} \cdot \bx + k c_p t),
$$
where 
\begin{align*}
\mathbf R_s & = [-\frac{2}{3}\mu, \frac{2}{3}\mu, 0, 0, \frac{1}{3}\mu, -\frac{1}{3}\mu, -\frac{1}{\sqrt 3}c_s, \frac{1}{\sqrt 3}c_s, 0]^T, \\
\mathbf R_p & = [\lambda + \frac{2}{3}\mu, \lambda + \frac{2}{3}\mu, \lambda + \frac{2}{3}\mu, \frac{2}{3}\mu, \frac{2}{3}\mu, 
\frac{2}{3}\mu, -\frac{1}{\sqrt 3}c_p, -\frac{1}{\sqrt 3}c_p, -\frac{1}{\sqrt 3}c_p]^T
\end{align*}
 and $\mathbf{k} = k \mathbf n, k = -{2\sqrt 3\pi }$. Similarly, we consider periodic boundary condition and $u_0(\bx) = u(0,\bx)$ as initial condition.
 \end{exa}
We present the numerical results at $T=1.$ In Table \ref{elastic_3d}, we get at least $(k+1/2)$-th order of accuracy for the solution $\bu(t, \bx)$.

\begin{table}[htp]
	
	\caption{$L^2$ errors and orders of accuracy for Example \ref{ex:system_elastic_3d} at $T=1$. $N$ denotes mesh level, $h_N$ is the size of the smallest mesh in each direction, $k$ is the polynomial order, $d$ is the dimension.  $L^2$ order is calculated with respect to  $h_N$. $d=3$. 
	}
	\centering
	\begin{tabular}{|c|c|c c|c c|c c|}
		\hline
		& & $L^2$ error & order& $L^2$ error & order & $L^2$  error & order\\
		\hline
		$N$& $h_N$&   \multicolumn{2}{c|}{$ k=1$}    &     \multicolumn{2}{c|}{$ k=2$}  &     \multicolumn{2}{c|}{$ k=3$}  \\
		\hline
		4 &$1/16$ &	 2.49E+00	&	--       &     4.93E-02	&	--     &	8.91E-04	&	--   \\
		5 &$1/32$ &  7.70E-01	&	1.69     &	   8.17E-03	&	2.59   &	8.66E-05	&	3.36   \\
		6 &$1/64$ &	 1.76E-01	&	2.13     &	   1.59E-03	&	2.36   &	7.12E-06	&	3.60   \\
		7 &$1/128$&	 4.27E-02   &	2.04     &	   2.79E-04	&	2.51   &	5.42E-07	&	3.72 	  \\

		\hline
		
	\end{tabular}
	\label{elastic_3d}
\end{table}

\section{Conclusions and future work}
\label{sec:conclu}

In this work, we develop sparse grid CDG schemes for linear transport problems. We construct sparse finite element space on primal and dual meshes for periodic and non-periodic problems. A new hierarchical representation of the piecewise polynomials is introduced and analyzed for non-periodic problems on the dual mesh. Compared with CDG scheme, the method is shown to be efficient for high dimensional problems. Compared with sparse grid DG scheme, the method proposed allows larger CFL numbers and avoid the evaluations of numerical fluxes. We show that for scalar equation with constant coefficients, the scheme shares similar $L^2$ stability property as the standard CDG scheme.   $L^2$ convergence rate is proved to be of $O(\abs{\log h}^d h^k)$ where $h$ is the most refined mesh in each direction. Numerical results are provided validating performance of the methods. In particular, the convergence order seems higher than the theoretically predicted rate, which suggests that new projection techniques  may be needed.
Other future work includes  detailed study of CFL conditions, and applications and extensions to nonlinear and nonsmooth problems.

\bibliographystyle{abbrv}
\bibliography{ref_sparse_VM,refer,cdg,ref_cheng,ref_cheng_2,n}

\end{document}